\documentclass{article}
\usepackage{amssymb,amsmath,bm,geometry,graphics,graphicx,url,color}

\def\no{\noindent}
\def\pmatrix{\left(\begin{array}}
\def\endpmatrix{\end{array}\right)}

\def\I{{\cal I}}

\newtheorem{theo}{Theorem}
\newtheorem{lem}{Lemma}

\newtheorem{rem}{Remark}
\newtheorem{defi}{Definition}

\newcommand{\norm}[1]{\left\Vert#1\right\Vert}

\title{Exponential collocation methods for conservative or dissipative systems}

\author{Bin Wang\,
\footnote{School of Mathematical Sciences, Qufu Normal University,
Qufu 273165, P.R. China; Mathematisches Institut, University of
T\"{u}bingen, Auf der Morgenstelle 10, 72076 T\"{u}bingen, Germany.
The research is supported in part by the Alexander von Humboldt
Foundation and by the Natural Science Foundation of Shandong
Province (Outstanding Youth Foundation) under Grant ZR2017JL003.
E-mail:~{\tt wang@na.uni-tuebingen.de} } \and Xinyuan
Wu\thanks{School of Mathematical Sciences, Qufu Normal University,
Qufu 273165, P.R. China; Department of Mathematics, Nanjing
University, Nanjing 210093, P.R. China. The research is supported in
part by the National Natural Science Foundation of China under Grant
11671200. E-mail:~{\tt xywu@nju.edu.cn}} }

\begin{document}
\maketitle

\begin{abstract} In this paper, we propose and analyse a novel class of exponential
collocation methods for solving conservative or dissipative systems
based on exponential integrators and collocation methods. It is
shown that these novel methods can be of arbitrarily high order and
exactly or nearly preserve first integrals or Lyapunov functions. We
also consider order estimates of the new methods. Furthermore, we
explore and discuss the application of our methods in important
stiff gradient systems, and it turns out that our methods are
unconditionally energy-diminishing and strongly damped even for very
stiff gradient systems. Practical examples of the new methods are
derived and the efficiency and superiority are confirmed and
demonstrated by three numerical experiments including a nonlinear
Schr\"{o}dinger equation. As a byproduct  of this paper,
arbitrary-order trigonometric/RKN   collocation methods are also
presented and analysed for second-order highly oscillatory/general
systems. The paper is accompanied by numerical results that
demonstrate the great potential of this work.
\medskip
\no{\bf Keywords:} exponential integrators, energy-preserving
algorithms, first integral, Lyapunov function, finite element
methods, collocation methods

\medskip
\no{\bf MSC:} 65L05, 65L60,  65P10
\end{abstract}
\section{Introduction}
 In this paper, we are  concerned with systems of ordinary
differential equations (ODEs) of the form
\begin{equation}\label{IVPPP}
y^{\prime}(t)=Q \nabla H(y(t)),\quad
y(0)=y_{0}\in\mathbb{R}^{d},\quad t\in[0,T],
\end{equation}
where $Q$ is an  invertible and $d \times d$ real matrix, and $H:
\mathbb{R}^{d}\rightarrow\mathbb{R}$ is defined  by
\begin{equation}\label{H}
 H(y)=  \dfrac{1}{2}y^{\intercal}My+V(y).
\end{equation}
Here $M$ is a $d \times d$ symmetric real matrix, and $V:
\mathbb{R}^{d}\rightarrow\mathbb{R}$ is a differentiable function.

It is important to note that the system \eqref{IVPPP}  exhibits
 remarkable  geometrical/physical  structures. If the matrix $Q$ is
skew symmetric, then \eqref{IVPPP} is a conservative system with the
first integral $H$: i.e.,
$$H(y(t))\equiv H(y_0)\qquad \textmd{for}\ \    \textmd{any}\ \  t\geq0.$$
When the matrix  $Q$ is negative semidefinite, then \eqref{IVPPP} is
a dissipative system with the Lyapunov function $H$: i.e.,
$$H(y(t_2))\leq H(y(t_1))\quad \textmd{if}\ \ t_2\geq t_1.$$
In this paper, we  call $H$  energy  for both cases in a broad
sense.  The objective of this paper is to design and analyse a class
of novel arbitrary-order exponential energy-preserving collocation
methods  to preserve first integrals or Lyapunov functions of the
conservative/dissipative system \eqref{IVPPP}.

For brevity,    let
\begin{equation*}\label{AG}
A=QM,\ \ g(y(t))=Q\nabla V(y(t)).
\end{equation*}
 We then  rewrite the system \eqref{IVPPP} as
\begin{equation}\label{IVP1}
y^{\prime}(t)=Ay(t)+g(y(t)),\quad y(0)=y_{0}\in\mathbb{R}^{d}.
\end{equation}
As is known, the exact solution of \eqref{IVPPP} or \eqref{IVP1} can
be represented by the variation-of-constants formula (the Duhamel
Principle)
  \begin{equation}
 y(t)= e^{tA}y_0+  t\int_{0}^1 e^{(1-\tau)tA}g(y(\tau t))d\tau.\\
\label{VOC}%
\end{equation}

The system \eqref{IVPPP} or \eqref{IVP1} plays a prominent role in a
wide range of applications  in physics and engineering, inclusive of
mechanics, astronomy, molecular dynamics, and in problems of wave
propagation in classical and quantum physics (see, e.g.
\cite{hairer2006,Hochbruck2010,xinyuanbook2018,wu2013-book}). Some
highly oscillatory problems and semidiscrete PDEs  such as
semilinear Schr\"{o}dinger equations all fit  this form. One
important example of them  is the multi-frequency highly oscillatory
Hamiltonian systems with the following Hamiltonian
  \begin{equation}H(q,p)=\frac{1}{2}p^{\intercal}\bar{M}^{-1}p+  \frac{1}{2} q^{\intercal}\bar{K} q+U
(q),\label{h os H}%
\end{equation}
where $\bar{K}$ is a symmetric positive semi-definite stiffness
matrix, $\bar{M}$ is a symmetric positive definite mass matrix, and
$U(q)$ is a smooth potential with moderately bounded derivatives.

 In recent decades,  exponential integrators have been widely
 investigated  and developed   as an efficient approach to
integrating
 \eqref{IVP1},  and we refer the
 reader to
 \cite{Berland-2005,Butcher09,Caliari-2009,Calvo-2006,Cano13,Celledoni-2008,Einkemmer-2017,Grimm2006,Hochbruck2005ANM,Hochbruck2009,Ostermann2006,wang2017-JCM,wu-2012-BIT}
for example. Exponential integrators make well use of the
variation-of-constants formula \eqref{VOC},  and their performance
has been evaluated by a range of test problems. A systematic survey
of exponential integrators is referred to \cite{Hochbruck2010}.
However,  apart from symplectic exponential integrators (see, e.g.
\cite{JCP2017_Mei_Wu}),  most  existing publications dealing with
exponential integrators focus on the construction and analysis of
the schemes and never consider deriving energy-preserving
exponential integrators to preserve the first integrals/Lyapunov
functions.
 Energy-preserving exponential integrators, especially
higher-order schemes of them have not been  well researched yet  in
the literature.

On the other hand, various  effective energy-preserving methods have
been proposed and researched for \eqref{IVP1} in the special case of
$A =0$, such as  the average vector field  (AVF) method
   \cite{Celledoni2010,Celledoni14,Quispel08}, discrete gradient (DG) methods
\cite{McLachlan14,McLachlan99,WWIMA2018},    Hamiltonian Boundary
Value Methods (HBVMs) \cite{Brugnano2010,Brugnano2012b},     the
Runge-Kutta-type energy-preserving collocation (RKEPC) methods
 \cite{Cohen-2011,Hairer2010}, time finite elements (TFE) methods
  \cite{Betsch00,Betsch00-1,Hansbo01,Li_Wu(na2016),Tang2012}, and energy-preserving
exponentially-fitted (EPEF) methods
\cite{Miyatake2014,Miyatake2015}. Some numerical methods preserving
Lyapunov functions have also been studied for \eqref{IVP1} with $A
=0$  (see, e.g. \cite{Calvo10,hairer2013,Mclachlan-98}). It is noted
that all these methods are constructed and studied for the special
case $A=0$ and thus they
 do not take  advantage of the structure brought by  the linear term  $Ay$ in
the system \eqref{IVP1}. These methods  could  be applied to
\eqref{IVP1} with $A\neq0$ if the right-hand side of \eqref{IVP1} is
considered as a whole function: i.e., $y'=f(y)\equiv Ay-g(y)$.

 Recently, in order to take advantage the structure of the underlying system and
preserve its energy   simultaneously,
 a novel energy-preserving method has
been studied in \cite{wang2012-1,wu2013-JCP} for second-order ODEs
and a new energy-preserving exponential scheme for the conservative
or dissipative system has been researched  in \cite{Li_Wu(sci2016)}.
However, these two kinds of methods are both based on the  AVF
methods  and thence they  are  only  of order two, in general. This
results in insufficiency  to deal with some practical problems for
high-precision  numerical simulations in sciences and engineering.

On the basis of the facts stated above,  this paper is devoted to
deriving and analysing  novel exponential  collocation methods. For
this purpose, we make  well use of
 the variation-of-constants formula  and the structure
introduced by the underlying system. In such a way,  these
exponential integrators can exactly or nearly preserve the first
integral or the Lyapunov function of \eqref{IVPPP}.

The paper is organised as follows. We first formulate the
exponential  collocation methods for first-order ODEs \eqref{IVPPP}
in the next section. As a byproduct, the trigonometric/RKN
collocation methods for second-order systems are presented in
Section \ref{second-system}. Section \ref{preserving analysis} pays
attention to showing  that the novel methods preserve exactly or
nearly first integrals or Lyapunov functions. From Section
\ref{existence} to Section \ref{algebraic order}, the properties of
the methods  are discussed in detail including existence and
uniqueness, and algebraic order.   In Section \ref{gradient
systems}, we discuss the application of our methods to stiff
gradient systems. Illustrative examples of the new methods are
derived in Section \ref{methods} and numerical experiments are
presented in Section \ref{Numerical experiments}.
 Finally, this paper ends with  some concluding remarks  and
 discussions.

\section{Formulation of  new   methods}\label{fomulation of methods}
 Following \cite{Li_Wu(na2016)}, we define the finite-dimensional
function spaces $Y_{h}$ as follows:
\begin{equation}\label{Yh}\begin{aligned}
Y_{h}&=\text{span}\left\{\tilde{\varphi}_{0}(\tau),\ldots,\tilde{\varphi}_{r-1}(\tau)\right\}\\
&=\left\{\tilde{w} :
\tilde{w}(\tau)=\sum_{i=0}^{r-1}\tilde{\varphi}_{i}(\tau)W_{i},\
\tau\in [0,1],\ W_{i}\in\mathbb{R}^{d}\right\},
\end{aligned}
\end{equation}
where $\{\tilde{\varphi}_{i}\}_{i=0}^{r-1}$ are supposed to be
linearly independent on $I$ and sufficiently smooth.  It is noted
that the notation $\tilde{\varphi}_{i} (\tau)$ is referred to
$\varphi_{i}(\tau h)$ for all the functions $\varphi_{i}$ throughout
this paper. With this definition, we consider another
finite-dimensional function space $X_{h}$  such that $\tilde{w}' \in
Y_h$
 for any $\tilde{w}\in X_h$.

We  introduce  the idea of  the formulation of  methods. Find
$\tilde{u}(\tau)$ with $\tilde{u}(0)=y_{0}$, satisfying that
\begin{equation}\label{projection} \tilde{u}^{\prime}(\tau)=A
\tilde{u}(\tau)+  \mathcal{P}_{h}g(\tilde{u}(\tau)),
\end{equation}
where the projection operation $\mathcal{P}_{h}$  is given by (see
\cite{Li_Wu(na2016)})
\begin{equation}\label{DEF}
\langle
\tilde{v}(\tau),\mathcal{P}_{h}\tilde{w}(\tau)\rangle=\langle
\tilde{v}(\tau),\tilde{w}(\tau)\rangle\quad \text{for any}\ \
\tilde{v}(\tau)\in Y_{h}
\end{equation}
and the inner product $\langle\cdot,\cdot\rangle$ is defined by (see
\cite{Li_Wu(na2016)})
\begin{equation*}
\langle w_{1},w_{2}\rangle=\langle
w_{1}(\tau),w_{2}(\tau)\rangle_{\tau}=\int_{0}^{1}w_{1}(\tau)\cdot
w_{2}(\tau)d\tau.
\end{equation*}

 With regard to  the projection operation $\mathcal{P}_{h}$, we have
the following property which is needed in this paper.
\begin{lem}\label{proj lem}(See \cite{Li_Wu(na2016)}) The projection $\mathcal{P}_{h}\tilde{w}$ can be explicitly expressed as
\begin{equation*}
\mathcal{P}_{h}\tilde{w}(\tau)=\langle
P_{\tau,\sigma},\tilde{w}(\sigma)\rangle_{\sigma},
\end{equation*}
where
\begin{equation}\label{explicit}
\begin{aligned}
&P_{\tau,\sigma}=(\tilde{\varphi}_{0}(\tau),\ldots,\tilde{\varphi}_{r-1}(\tau))\Theta^{-1}(\tilde{\varphi}_{0}(\sigma),\ldots,\tilde{\varphi}_{r-1}(\sigma))^{\intercal},\\
&\Theta=(\langle\tilde{\varphi}_{i}(\tau),\tilde{\varphi}_{j}(\tau)\rangle)_{0\leq
i,j\leq r-1}.\end{aligned}
\end{equation}
When $h$ tends to 0, the limit of $P_{\tau,\sigma}$   exists. If
$P_{\tau,\sigma}$ is computed by    a standard orthonormal basis
$\left\{\tilde{\psi}_{0},\ldots,\tilde{\psi}_{r-1}\right\}$
 of $Y_{h}$ under the inner product $\langle\cdot,\cdot\rangle$,
then   $\Theta$ is an identity matrix and $P_{\tau,\sigma}$ has a
simpler expression:
\begin{equation}\label{PCOEFF}
P_{\tau,\sigma}=\sum_{i=0}^{r-1}\tilde{\psi}_{i}(\tau)\tilde{\psi}_{i}(\sigma).
\end{equation}
\end{lem}

  Because  $\tilde{u}(\tau)=u(\tau h)$,
\eqref{projection}  can be expressed
\begin{equation*}
 \begin{aligned}\label{odes}
u^{\prime}(\tau h)= A u(\tau h)+\langle P_{\tau,\sigma},g(u(\sigma
h))\rangle_{\sigma}.
\end{aligned}
\end{equation*}
 Applying the variation-of-constants formula \eqref{VOC} to
\eqref{projection},  we obtain
\begin{equation}\label{uu}
 \begin{aligned} &\tilde{u}(\tau)=u(\tau h)=e^{\tau h   A}y_{0}+\tau h    \int_{0}^1
e^{(1-\xi)\tau h   A}\langle P_{\xi
\tau,\sigma},g(u(\sigma h))\rangle_{\sigma}d\xi\\
=&e^{\tau h A}y_{0}+\tau h    \int_{0}^1 e^{(1-\xi)\tau h A}\langle
P_{\xi \tau,\sigma},g(\tilde{u}(\sigma))\rangle_{\sigma}d\xi.
\end{aligned}
\end{equation}
 Inserting \eqref{PCOEFF} into \eqref{uu} yields
\begin{equation*}\label{u meth}
 \begin{aligned} \tilde{u}(\tau)&=e^{\tau h A}y_{0}+\tau h  \int_{0}^1
e^{(1-\xi)\tau h A} \int_{0}^1 \sum_{i=0}^{r-1}\tilde{\psi}_{i}(\xi
\tau)\tilde{\psi}_{i}(\sigma)g(\tilde{u}(\sigma)) d\sigma d\xi\\
&=e^{\tau h A}y_{0}+\tau h  \int_{0}^1 \sum_{i=0}^{r-1}
 \int_{0}^1 e^{(1-\xi)\tau h A}\tilde{\psi}_{i}(\xi \tau)d\xi \tilde{\psi}_{i}(\sigma)g(\tilde{u}(\sigma))
d\sigma.
\end{aligned}
\end{equation*}
We are now  in a position  to present the scheme of the novel
methods.
\begin{defi}\label{def EEPCr}
  An  exponential   collocation method  for solving the
 system \eqref{IVPPP} or
\eqref{IVP1}  is defined  as follows:
\begin{equation}\label{EEPCr}
\begin{aligned} &\tilde{u}(\tau)=e^{\tau h A}y_{0}+\tau h
\int_{0}^1 \bar{A}_{\tau,\sigma}(A) g(\tilde{u}(\sigma))
d\sigma,\qquad y_{1}=\tilde{u}(1),
\end{aligned}
\end{equation}
where $h$ is  a  stepsize,
\begin{equation}\label{Aexplicit}
\bar{A}_{\tau,\sigma}(A)= \int_{0}^1 e^{(1-\xi)\tau h A}P_{\xi
\tau,\sigma}d\xi=\sum_{i=0}^{r-1}
 \int_{0}^1 e^{(1-\xi)\tau h A}\tilde{\psi}_{i}(\xi \tau)d\xi
\tilde{\psi}_{i}(\sigma),
\end{equation}
and $\left\{\tilde{\psi}_{0},\ldots,\tilde{\psi}_{r-1}\right\}$ is
 a standard orthonormal basis
 of $Y_{h}$. We denote the method as ECr.
\end{defi}

\begin{rem}Once the stepsize $h$ is chosen, the method
\eqref{EEPCr} approximates  the solution of \eqref{IVPPP} in the
time interval $I_0$. Obviously, the obtained result   can be
considered as the initial condition for a new initial value problem
and it can be approximated in the  next  time interval $I_1$. In
general, the method  can be extended to the approximation of the
solution in the interval $[0,T]$.
\end{rem}

\begin{rem}
It can be observed  that  the  ECr method  \eqref{EEPCr}
exactly integrates  the homogeneous linear system $y'=Ay$.  
The scheme \eqref{EEPCr}  can be classified into  the category of
exponential integrators (which can be thought of as continuous-stage
exponential integrators).   This  is an interesting and important
class of numerical methods for first-order ODEs (see, e.g.
\cite{Hochbruck1998,Hochbruck2005ANM,Hochbruck2005,Hochbruck2010,Hochbruck2009}).
   In
\cite{Li_Wu(sci2016)}, the authors researched  a new
energy-preserving exponential scheme for the conservative or
dissipative system.  Here we note that  its order is   of only two
since
 this  scheme combines  the ideas of DG and AVF methods. We have
proposed a kind of arbitrary-order exponential Fourier collocation
methods in \cite{wang2017-JCM}.  However,  those methods cannot
preserve energy exactly.   Fortunately,  we will show that the novel
ECr method \eqref{EEPCr}  can be of arbitrarily high order and can
exactly or  nearly preserve energy, which is different from the
existing exponential integrators in the literature. This feature is
significant  and makes our methods be more efficient and robust.
\end{rem}

\begin{rem}
Consider  $M=0$ and
 $Q=\left(\begin{array}{cc}O_{d_{1}\times
d_{1}}&-I_{d_{1}\times d_{1}}\\I_{d_{1}\times d_{1}}&O_{d_{1}\times
d_{1}}\end{array}\right),$ which means that \eqref{IVPPP} is a
Hamiltonian system.  In this special case,  if we   choose  $X_h$
and $Y_h$ as
\begin{equation*}
Y_h=\text{span}\left\{\tilde{\varphi}_{0}(\tau),\ldots,\tilde{\varphi}_{r-1}(\tau)\right\},\quad
X_h=\text{span}\left\{1,\int_{0}^{\tau}\tilde{\varphi}_{0}(s)ds,\ldots,\int_{0}^{\tau}
\tilde{\varphi}_{r-1}(s)ds\right\},
\end{equation*}
then the  ECr method \eqref{EEPCr} becomes the following
energy-preserving Runge-Kutta type collocation methods
\begin{equation*}\label{EEPCr-li}
\begin{aligned} &\tilde{u}(\tau)= y_{0}+\tau h
\int_{0}^1 \int_{0}^1  P_{\xi \tau,\sigma}d\xi g(\tilde{u}(\sigma))
d\sigma,\qquad y_{1}=\tilde{u}(1),
\end{aligned}
\end{equation*}
 which  yields  the functionally-fitted   TFE method
derived in \cite{Li_Wu(na2016)}.  Moreover,  under the above choices
of $M$ and $Q$, if $Y_h$ is particularly generated by the shifted
Legendre polynomials on $[0,1]$,
 then the
ECr method \eqref{EEPCr} reduces to  the  RKEPC
  method of order $2r$ given in
\cite{Hairer2010}  or   HBVM$(\infty,r)$ presented in
\cite{Brugnano2010}. Consequently, the ECr method \eqref{EEPCr} can
be regarded as a generalisation of these existing methods in the
literature.

\end{rem}

\section{Methods for second-order  highly oscillatory   ODEs } \label{second-system}

We consider the following second-order highly oscillatory problems
 \begin{equation}
q^{\prime\prime}(t)-N q^{\prime}(t)+\Omega q(t)=-\nabla U(q(t)),
\qquad
q(0)=q_0,\ \ q'(0)=q'_0,\qquad t\in[0,T],\label{2oprob}%
\end{equation}
where   $N$ is a symmetric negative semidefinite matrix, $\Omega$ is
a symmetric positive semidefinite matrix, and $U:
\mathbb{R}^{d}\rightarrow\mathbb{R}$ is a differential function. By
introducing $p=q^{\prime}$, \eqref{2oprob} can be transformed into
\begin{equation}\begin{aligned}& \left(
                   \begin{array}{c}
                     q \\
                      p \\
                   \end{array}
                 \right)
'=\left(
    \begin{array}{cc}
      0 & I \\
      -I &N \\
    \end{array}
  \right)\nabla H(q,p)
\label{2-1oprob}%
\end{aligned}\end{equation}
with
\begin{equation}\label{H2}H(q,p) =
\frac{1}{2}p^{\intercal}p+  \frac{1}{2} q^{\intercal}\Omega q+U
(q).\end{equation} This is exactly the same as  the problem
\eqref{IVPPP}.  Since $N$ is   symmetric negative semidefinite,
\eqref{2-1oprob} is a dissipative system with the
 Lyapunov function \eqref{H2}.
 In the particular case $N = 0$, \eqref{2-1oprob} becomes a
conservative Hamiltonian system with the first integral \eqref{H2}.
 This is an  important  highly oscillatory system which  has been investigated by many
researchers (see, e.g.
 \cite{Cohen-2005,franco2006,B.garcia1999,Hochbruck1999,Iserles 02,wang2017-ANM,wu2017-JCAM,wang2017-Cal,wu2010-1,wu2013-book}).

Applying the ECr method \eqref{EEPCr} to \eqref{2-1oprob} yields
{the trigonometric  collocation method   for second-order highly
oscillatory problems.  In particular, for Hamiltonian systems
\begin{equation}\label{os2oprob}q^{\prime\prime}(t)+\Omega q(t)=-\nabla U(q(t)),\end{equation}
the  case of $N=0$ in \eqref{2oprob}, the ECr method \eqref{EEPCr}
becomes  the following form.

\begin{defi}
 The
 trigonometric   collocation (denoted by TCr) method
 for \eqref{os2oprob} is defined as:
\begin{equation}\label{EEPCr of erkn}
\left\{\begin{aligned} & \tilde{q}(\tau)= \phi_{0} (K )q_0 +\tau
h\phi_{1} (K ) p_0 -\tau^2 h^2  \int_{0}^1 \mathcal{A}_{\tau,\sigma}(K) f(\tilde{q}(\sigma))d\sigma,\qquad   q_{1} =  \tilde{q}(1),\\
& \tilde{p}(\tau)= -\tau h\Omega\phi_{1}  (K)q_0 +\phi_{0} (K) p_0
-\tau h  \int_{0}^1 \mathcal{B}_{\tau,\sigma}(K)
f(\tilde{q}(\sigma))d\sigma,\ \ \quad p_{1}=\tilde{p}(1),
\end{aligned}\right.
\end{equation}
where $K=\tau^2 h^{2}\Omega,$ $f(q)=\nabla U(q),$  $
\phi_{i}(K):=\sum\limits_{l=0}^{\infty}\dfrac{(-1)^{l}K^{l}}{(2l+i)!}$
 for $
 i=0,1,\ldots,$ and
\begin{equation}\label{AB of erkn}
\begin{aligned} &\mathcal{A}_{\tau,\sigma}(K)= \sum_{i=0}^{r-1}
 \int_{0}^1(1-\xi) \phi_{1} \big((1-\xi)^2K\big)\tilde{\psi}_{i}(\xi
 \tau)d\xi \tilde{\psi}_{i}(\sigma),\\
& \mathcal{B}_{\tau,\sigma}(K)= \sum_{i=0}^{r-1}
 \int_{0}^1 \phi_{0} \big((1-\xi)^2K\big ) \tilde{\psi}_{i}(\xi \tau)d\xi
\tilde{\psi}_{i}(\sigma).
 \end{aligned}
\end{equation}
\end{defi}

\begin{rem}
In  \cite{wang-2016}, the authors developed and researched a novel
type of trigonometric Fourier collocation methods   for second-order
ODEs $q^{\prime\prime}(t)+Mq(t)=f(q(t))$.  However, as shown  in
\cite{wang-2016}, those methods cannot preserve the energy exactly.
From the analysis  to be presented  in this paper, it turns out
 that  the trigonometric collocation scheme  \eqref{EEPCr of erkn}
developed here  can attain arbitrary algebraic order and can exactly
or nearly preserve the energy of \eqref{H2}.
\end{rem}

\begin{rem}
It is remarked that the  multi-frequency highly oscillatory
Hamiltonian system  \eqref{h os H} is a kind of second-order system
$q^{\prime\prime}(t)+\bar{M}^{-1}\bar{K} q(t)=-\bar{M}^{-1}\nabla
U(q(t))$ and applying the ECr method \eqref{EEPCr} to it leads to
the TCr  method \eqref{EEPCr of erkn} with $K=\tau^2
h^{2}\bar{M}^{-1}\bar{K}$ and  $f(q)=\bar{M}^{-1}\nabla U(q).$
\end{rem}

  In the special case of $N=0$ and $\Omega=0,$  the
system \eqref{2oprob}  reduces to  the  conventional  second-order
ODEs
\begin{equation}
q^{\prime\prime}(t)=-\nabla U(q(t)), \qquad
q(0)=q_0,\ \ q'(0)=q'_0,\qquad t\in[0,T].\label{ge2oprob}%
\end{equation}
Then the TCr method has the  following form.
\begin{defi}
 A   TCr method  for solving \eqref{ge2oprob} is defined as
\begin{equation}\label{EEPCr of rkn}
\left\{\begin{aligned} & \tilde{q}(\tau)=  q_0 +\tau
h  p_0 -\tau^2 h^2  \int_{0}^1 \bar{\mathcal{A}}_{\tau,\sigma}  \nabla U(\tilde{q}(\sigma))d\sigma,\ \ \ q_{1} =  \tilde{q}(1),\\
& \tilde{p}(\tau)=   p_0
-\tau h  \int_{0}^1 \bar{\mathcal{B}}_{\tau,\sigma}  \nabla U(\tilde{q}(\sigma))d\sigma,\ \quad \qquad \ \ \ \quad p_{1}=\tilde{p}(1),\\
\end{aligned}\right.
\end{equation}
where
\begin{equation}\label{AB of rkn}
\begin{aligned} &\bar{\mathcal{A}}_{\tau,\sigma} = \sum_{i=0}^{r-1}
 \int_{0}^1(1-\xi)  \tilde{\psi}_{i}(\xi
 \tau)d\xi \tilde{\psi}_{i}(\sigma),\ \bar{\mathcal{B}}_{\tau,\sigma}= \sum_{i=0}^{r-1}
 \int_{0}^1   \tilde{\psi}_{i}(\xi \tau)d\xi
\tilde{\psi}_{i}(\sigma).
 \end{aligned}
\end{equation}
This scheme looks like  a continuous-stage  RKN method, and is
denoted by
   RKNCr in this paper.

\end{defi}

\section{Energy-preserving analysis}\label{preserving analysis}
In this section, we analyse the energy-preserving property of the
ECr methods.

\begin{theo}\label{EP}
If $Q$ is skew symmetric and   $\tilde{u}(\tau) \in X_h$, the first
integral $H$ \eqref{H} of the conservative system \eqref{IVPPP} can
be preserved exactly by the ECr method \eqref{EEPCr}: i.e.,
$H(y_{1})=H(y_{0}).$ If $\tilde{u}(\tau) \notin X_h$, the ECr method
\eqref{EEPCr} approximately preserves the   energy $H$ with the
following accuracy $H(y_{1})=H(y_{0})+\mathcal{O}(h^{2r+1}).$
\end{theo}
\begin{proof}
 We begin with the first part of this proof under the
assumption that $Q$ is skew symmetric and   $\tilde{u}(\tau) \in
X_h$. From $\tilde{u}(\tau)\in X_{h}$, it follows that
$\tilde{u}^{\prime}(\tau)\in Y_{h}$ and
$Q^{-1}\tilde{u}^{\prime}(\tau)\in Y_{h}$. Then,   in the light of
\eqref{DEF},  we obtain
\begin{equation*}
\begin{aligned}
\int_{0}^{1}\tilde{u}^{\prime}(\tau)^{\intercal}(Q^{-1})^{\intercal}\tilde{u}^{\prime}(\tau)d\tau
&=\int_{0}^{1}\tilde{u}^{\prime}(\tau)^{\intercal}(Q^{-1})^{\intercal}\big(A\tilde{u}(\tau)+\mathcal{P}_{h}g(\tilde{u}(\tau))\big)d\tau\\
&=\int_{0}^{1}\tilde{u}^{\prime}(\tau)^{\intercal}(Q^{-1})^{\intercal}\big(A\tilde{u}(\tau)+g(\tilde{u}(\tau))\big)d\tau.
\end{aligned}
\end{equation*}
 Here  $Q$ is skew symmetric,  so does $Q^{-1}$. We then have
$$0=\int_{0}^{1}\tilde{u}^{\prime}(\tau)^{\intercal}(Q^{-1})^{\intercal}\tilde{u}^{\prime}(\tau)d\tau
=-\int_{0}^{1}\tilde{u}^{\prime}(\tau)^{\intercal}Q^{-1}\big(A\tilde{u}(\tau)+g(\tilde{u}(\tau))\big)d\tau.$$
On the other hand,  it is clear that
\begin{equation*}
\begin{aligned}
&H(y_{1})-H(y_{0})=\int_{0}^{1}\frac{d}{d\tau}H(\tilde{u}(\tau))d\tau
=h\int_{0}^{1}\tilde{u}^{\prime}(\tau)^{\intercal}\nabla H(\tilde{u}(\tau))d\tau.\\
\end{aligned}
\end{equation*}
It follows from  \eqref{IVPPP} and \eqref{IVP1} that
$$\nabla H(\tilde{u}(\tau))=Q^{-1}\big(A\tilde{u}(\tau)+g(\tilde{u}(\tau))\big).$$
Therefore, we obtain
\begin{equation*}
\begin{aligned}
&H(y_{1})-H(y_{0})=h\int_{0}^{1}\tilde{u}^{\prime}(\tau)^{\intercal}Q^{-1}\big(A\tilde{u}(\tau)+g(\tilde{u}(\tau))\big)d\tau=h\cdot0=
0.\\
\end{aligned}
\end{equation*}

We next prove the second part of this theorem under the assumption
that $\tilde{u}(\tau) \notin X_h$. With the above analysis for the
first part of the proof, we have
\begin{equation*}
\begin{aligned}
&H(y_{1})-H(y_{0})=h\int_{0}^{1}\tilde{u}^{\prime}(\tau)^{\intercal}Q^{-1}\big(A\tilde{u}(\tau)+g(\tilde{u}(\tau))\big)d\tau\\
=&h\int_{0}^{1}\tilde{u}^{\prime}(\tau)^{\intercal}Q^{-1}\big(A\tilde{u}(\tau)+\mathcal{P}_{h}g(\tilde{u}(\tau))+g(\tilde{u}(\tau))-\mathcal{P}_{h}g(\tilde{u}(\tau))\big)d\tau\\
=&-h\int_{0}^{1}\tilde{u}^{\prime}(\tau)^{\intercal}(Q^{-1})^{\intercal}\tilde{u}^{\prime}(\tau)d\tau+
h\int_{0}^{1}\tilde{u}^{\prime}(\tau)^{\intercal}Q^{-1}\big(g(\tilde{u}(\tau))-\mathcal{P}_{h}g(\tilde{u}(\tau))\big)d\tau\\
=&
h\int_{0}^{1}\tilde{u}^{\prime}(\tau)^{\intercal}Q^{-1}\big(g(\tilde{u}(\tau))-\mathcal{P}_{h}g(\tilde{u}(\tau))\big)d\tau.
\end{aligned}
\end{equation*}
Concerning   Lemma 3.4 presented in \cite{Li_Wu(na2016)} and Lemma
\ref{lemma1}  proved in Section \ref{algebraic order}, one has
$\tilde{u}^{\prime}(\tau)=\mathcal{P}_{h}\tilde{u}^{\prime}(\tau)+\mathcal{O}(h^{r})$.
Therefore, one arrives at
\begin{equation*}
\begin{aligned}
H(y_{1})-H(y_{0}) =&
h\int_{0}^{1}\big(\mathcal{P}_{h}\tilde{u}^{\prime}(\tau)+\mathcal{O}(h^{r})\big)^{\intercal}Q^{-1}\big(g(\tilde{u}(\tau))-\mathcal{P}_{h}g(\tilde{u}(\tau))\big)d\tau\\
 =&
h\int_{0}^{1}\big(\mathcal{P}_{h}\tilde{u}^{\prime}(\tau)\big)^{\intercal}Q^{-1}\big(g(\tilde{u}(\tau))-\mathcal{P}_{h}g(\tilde{u}(\tau))\big)d\tau+\mathcal{O}(h^{2r+1})\\
 =&
h\int_{0}^{1}\big(\mathcal{P}_{h}\tilde{u}^{\prime}(\tau)\big)^{\intercal}Q^{-1}\big(g(\tilde{u}(\tau))-g(\tilde{u}(\tau))\big)d\tau+\mathcal{O}(h^{2r+1})=\mathcal{O}(h^{2r+1}),
\end{aligned}
\end{equation*}
where the result \eqref{error} is used.

  The proof is complete.
\end{proof}

\begin{theo}\label{DP}
If $Q$ is negative semidefinite and   $\tilde{u}(\tau) \in X_h$,
then the Lyapunov function $H$
 given by  \eqref{H} of the dissipative system \eqref{IVPPP} can be
preserved by the ECr method \eqref{EEPCr}; i.e., $H(y_{1})\leq
H(y_{0}).$ If $\tilde{u}(\tau) \notin X_h$,  it is true that
$H(y_{1})\leq H(y_{0})+\mathcal{O}(h^{2r+1}).$
\end{theo}
\begin{proof} According to the fact that $\int_{0}^{1}\tilde{u}^{\prime}(\tau)^{\intercal} Q^{-1}
\tilde{u}^{\prime}(\tau)d\tau \leq 0,$ this theorem can be proved in
a similar way  to   the proof of Theorem \ref{EP}.
  \end{proof}

\section{The existence, uniqueness
 and  smoothness}\label{existence}

In this section, we  focus on the study of  the existence and
uniqueness of $\tilde{u}(\tau)$ associated with the ECr method
\eqref{EEPCr}.


According to the Lemma 3.1 given in \cite{Hochbruck2005}, it is
easily  verified  that the coefficients $e^{\tau h A}$ and
$\bar{A}_{\tau,\sigma}(A)$  of our methods for $0\leq\tau\leq1$ and
$0\leq\sigma\leq1$   are  uniformly bounded.  We begin with assuming
that
\begin{equation*}
M_{k}=\max_{\tau,\sigma,h\in[0,1]}\norm{\frac{\partial^{k}\bar{A}_{\tau,\sigma}}{\partial
h^{k}}},\ \ C_{k}=\max_{\tau,h\in[0,1]}\norm{\frac{\partial^{k}
e^{\tau h A}}{\partial h^{k}}y_{0}}, \quad k=0,1,\ldots.
\end{equation*}
 Furthermore,   the $n$th-order
derivative of $g$ at $y$  is denoted  by $g^{(n)}(y).$
%
 We then  have the following result about the existence and
uniqueness of our methods.

\begin{theo}\label{eus}
  Let
$B(\bar{y}_{0},R)=\left\{y\in\mathbb{R}^{d} : ||y-\bar{y}_{0}||\leq
R\right\}$ and $$D_{n}=\max_{y\in B(\bar{y}_{0},R)}||g^{(n)}(y)||,\
n=0,1,\ldots,$$ where  $R$ is  a positive constant,
$\bar{y}_{0}=e^{\tau h A}y_{0}$, $||\cdot||=||\cdot||_{\infty}$ is
the maximum norm for vectors in $\mathbb{R}^{d}$ or the
corresponding induced norm for the multilinear maps $g^{(n)}(y)$. If
$h$ satisfies
\begin{equation}\label{cond2} 0\leq
h\leq
\kappa<\min\left\{\frac{1}{M_{0}D_{1}},\frac{R}{M_{0}D_{0}},1\right\},
\end{equation}
 then the ECr
method \eqref{EEPCr} has a unique solution $\tilde{u}(\tau)$ which
is smoothly dependent  on  $h$.
\end{theo}
\begin{proof}
Set $\tilde{u}_{0}(\tau)=\bar{y}_{0}$ and define
\begin{equation}\label{recur}
\tilde{u}_{n+1}(\tau)= e^{\tau h A}y_{0}+\tau h  \int_{0}^1
\bar{A}_{\tau,\sigma}(A) g(\tilde{u}_n(\sigma)) d\sigma, \quad
n=0,1,\ldots,
\end{equation}
which leads to a function series
$\{\tilde{u}_{n}(\tau)\}_{n=0}^{\infty}.$  We note that
$\lim\limits_{n\to\infty}\tilde{u}_{n}(\tau)$ is a solution  of the
TCr method \eqref{EEPCr}   if
$\left\{\tilde{u}_{n}(\tau)\right\}_{n=0}^{\infty}$ is uniformly
convergent, which will be shown by proving the uniform convergence
of the infinite series
$\sum_{n=0}^{\infty}(\tilde{u}_{n+1}(\tau)-\tilde{u}_{n}(\tau)).$

By induction and according to     \eqref{cond2} and \eqref{recur},
we get $ ||\tilde{u}_{n}(\tau)-\bar{y}_{0}||\leq R$ for
$n=0,1,\ldots.$ It then follows from \eqref{recur} that
\begin{equation*}
\begin{aligned}
&||\tilde{u}_{n+1}(\tau)-\tilde{u}_{n}(\tau)||\leq \tau h\int_{0}^{1}M_{0}D_{1}||\tilde{u}_{n}(\sigma)-\tilde{u}_{n-1}(\sigma)||d\sigma\\
& \leq  h\int_{0}^{1}M_{0}D_{1}||\tilde{u}_{n}(\sigma)-\tilde{u}_{n-1}(\sigma)||d\sigma \leq\beta||\tilde{u}_{n}-\tilde{u}_{n-1}||_{c},\quad{\beta=\kappa M_{0}D_{1},}\\
\end{aligned}
\end{equation*}
where $||\cdot||_{c}$ is the maximum norm for continuous functions
defined as $ ||w||_{c}=\max_{\tau\in[0,1]}||w(\tau)||$ for
  a continuous $\mathbb{R}^{d}$-valued function $w$ on
$[0,1]$.  Hence, we obtain
\begin{equation*}
||\tilde{u}_{n+1}-\tilde{u}_{n}||_{c}\leq\beta||\tilde{u}_{n}-\tilde{u}_{n-1}||_{c}
\end{equation*}
and
\begin{equation*}\label{recur3}
||\tilde{u}_{n+1}-\tilde{u}_{n}||_{c}\leq\beta^{n}||\tilde{u}_{1}-y_{0}||_{c}\leq\beta^{n}R,\quad
n=0,1,\ldots.
\end{equation*}
 From Weierstrass $M$-test and the fact that $\beta<1,$ it
immediately follows  that
$\sum_{n=0}^{\infty}(\tilde{u}_{n+1}(\tau)-\tilde{u}_{n}(\tau))$ is
uniformly convergent

If the ECr method \eqref{EEPCr}  has another solution
$\tilde{v}(\tau)$,  we then  obtain the following inequalities
 \begin{equation*}
||\tilde{u}(\tau)-\tilde{v}(\tau)||\leq
h\int_{0}^{1}||\bar{A}_{\tau,\sigma}(A)\big(g(\tilde{u}(\sigma))-g(\tilde{v}(\sigma))\big)||d\sigma\leq\beta||\tilde{u}-\tilde{v}||_{c},
\end{equation*}
and $
\norm{\tilde{u}-\tilde{v}}_{c}\leq\beta||\tilde{u}-\tilde{v}||_{c}.
$ This yields  $||\tilde{u}-\tilde{v}||_{c}=0$ and
$\tilde{u}(\tau)\equiv \tilde{v}(\tau)$. The existence and
uniqueness have been proved.

 With respect to  the result that  $\tilde{u}(\tau)$  is smoothly
dependent of $h$,  since each $\tilde{u}_{n}(\tau)$ is a
 smooth  function of $h$, we need only to prove that the series
$\left\{\frac{\partial^{k}\tilde{u}_{n}}{\partial
h^{k}}(\tau)\right\}_{n=0}^{\infty}$  is  uniformly convergent for
$k\geq1$. Differentiating  \eqref{recur} with respect to $h$ gives
\begin{equation}\label{recur2}
\begin{aligned}
\frac{\partial \tilde{u}_{n+1}}{\partial h}(\tau)=&\tau   Ae^{\tau h
A}y_0+\tau \int_{0}^{1}\Big(\bar{A}_{\tau,\sigma}(A)+h\frac{\partial
\bar{A}_{\tau,\sigma}}{\partial
h}\Big)g(\tilde{u}_{n}(\sigma))d\sigma \\
&+\tau
h\int_{0}^{1}\bar{A}_{\tau,\sigma}(A)g^{(1)}(\tilde{u}_{n}(\sigma))\frac{\partial
\tilde{u}_{n}}{\partial h}(\sigma)d\sigma, \end{aligned}
\end{equation}
which yields that
\begin{equation*}
\norm{\frac{\partial \tilde{u}_{n+1}}{\partial
h}}_{c}\leq\alpha+\beta\norm{\frac{\partial \tilde{u}_{n}}{\partial
h}}_{c},\quad \alpha=C_1+(M_{0}+\kappa M_{1})D_{0}.
\end{equation*}
By induction, it is easy to show that $\left\{\frac{\partial
\tilde{u}_{n}}{\partial h}(\tau)\right\}_{n=0}^{\infty}$ is
uniformly bounded:
\begin{equation}\label{bound}
\norm{\frac{\partial \tilde{u}_{n}}{\partial
h}}_{c}\leq\alpha(1+\beta+\ldots+\beta^{n-1})\leq\frac{\alpha}{1-\beta}=C^{*},\quad
n=0,1,\ldots.
\end{equation}
 It follows from \eqref{recur2}--\eqref{bound} that
\begin{equation*}
\begin{aligned}
&\norm{\frac{\partial \tilde{u}_{n+1}}{\partial h}-\frac{\partial
\tilde{u}_{n}}{\partial h}}_{c} \leq
 \tau \int_{0}^{1}(M_{0}+hM_{1})\norm{g(\tilde{u}_{n}(\sigma))-g(\tilde{u}_{n-1}(\sigma))}d\sigma\\
&+\tau
h\int_{0}^{1}M_{0}\Big(\norm{\big(g^{(1)}(\tilde{u}_{n}(\sigma))-g^{(1)}(\tilde{u}_{n-1}(\sigma))\big)\frac{\partial
\tilde{u}_{n}}{\partial h}(\sigma)}\\
&+\norm{g^{(1)}(\tilde{u}_{n-1}(\sigma))\Big(\frac{\partial
\tilde{u}_{n}}{\partial h}(\sigma)- \frac{\partial
\tilde{u}_{n-1}}{\partial
h}(\sigma)\Big)}\Big)d\sigma\leq\gamma\beta^{n-1}+
\beta\norm{\frac{\partial \tilde{u}_{n}}{\partial h}-\frac{\partial \tilde{u}_{n-1}}{\partial h}}_{c},\\
\end{aligned}
\end{equation*}
where $\gamma=(M_{0}D_{1}+\kappa M_{1}D_{1}+\kappa
M_{0}L_{2}C^{*})R,$ and $L_{2}$ is a constant satisfying
\begin{equation*}
||g^{(1)}(y)-g^{(1)}(z)||\leq L_{2}||y-z||,\quad\text{for\ \ $y,z\in
B(\bar{y}_{0},R)$}.
\end{equation*}
Therefore, the following result is obtained   by induction
\begin{equation*}
\norm{\frac{\partial \tilde{u}_{n+1}}{\partial h}-\frac{\partial
\tilde{u}_{n}}{\partial h}}_{c}\leq
n\gamma\beta^{n-1}+\beta^{n}C^{*},\quad n=1,2,\ldots.
\end{equation*}
 This shows  the   uniform convergence of
 $\sum_{n=0}^{\infty}(\frac{\partial \tilde{u}_{n+1}}{\partial
h}(\tau)-\frac{\partial \tilde{u}_{n}}{\partial h}(\tau))$ and then
 $\left\{\frac{\partial \tilde{u}_{n}}{\partial
h}(\tau)\right\}_{n=0}^{\infty}$ is uniformly convergent.

 Likewise,   it can be shown  that other function series
$\left\{\frac{\partial^{k}\tilde{u}_{n}}{\partial
h^{k}}(\tau)\right\}_{n=0}^{\infty}$ for $k\geq2$ are uniformly
convergent as well. Therefore, $\tilde{u}(\tau)$ is  smoothly
dependent  on $h$.
\end{proof}

\section{Algebraic order}\label{algebraic order}
In this section, we analyse the algebraic order of the ECr method
\eqref{EEPCr}.
 To express the dependence of the
solutions of $y'(t)=Ay(t)+g(y(t))$ on the initial values, we denote
by $y(\cdot,\tilde{t}, \tilde{y})$ the solution satisfying the
initial condition $y(\tilde{t},\tilde{t}, \tilde{y})=\tilde{y}$ for
any given $\tilde{t}\in[0,h]$ and set $\Phi(s,\tilde{t},
\tilde{y})=\frac{\partial y(s,\tilde{t}, \tilde{y})}{\partial
\tilde{y}}. $ Recalling the elementary theory of ODEs, we have the
following standard result
\begin{equation*}
\frac{\partial y(s,\tilde{t}, \tilde{y})}{\partial
\tilde{t}}=-\Phi(s,\tilde{t}, \tilde{y})\big(A\tilde{
y}+g(\tilde{y})\big).
\end{equation*}


In this section, for convenience,  an $h$-dependent function
$w(\tau)$ is called as regular if it can be expanded as $
w(\tau)=\sum_{n=0}^{r-1}w^{[n]}(\tau)h^{n}+\mathcal{O}(h^{r}), $
where $w^{[n]}(\tau)=\frac{1}{n!}\frac{\partial^{n}w(\tau)}{\partial
h^{n}}|_{h=0}$ is a vector-valued function with polynomial entries
of degrees $\leq n$.

It  can be deduced from   Proposition 3.3 in \cite{Li_Wu(na2016)}
that $P_{\tau,\sigma}$ is regular. Moreover, we can prove   the
following result.

\begin{lem}\label{lemma1}
The ECr method \eqref{EEPCr} gives a regular $h$-dependent
 function $\tilde{u}(\tau)$.
\end{lem}
\begin{proof}
By the result given in \cite{Li_Wu(na2016)}, we know that
$P_{\tau,\sigma}$ can be smoothly extended to $h=0$ by setting
$P_{\tau,\sigma}|_{h=0}=\lim\limits_{h\to0}P_{\tau,\sigma}(h)$.
Furthermore, it follows from   Theorem \ref{eus} that
$\tilde{u}(\tau)$ is  smoothly dependent on $h$. Therefore,
$\tilde{u}(\tau)$ and $\bar{A}_{\tau,\sigma}(A)$ can be expanded
with respect to $h$ at zero as follows:
\begin{equation*}
\tilde{u}(\tau)=\sum_{m=0}^{r-1}\tilde{u}^{[m]}(\tau)h^{m}+\mathcal{O}(h^{r}),\quad
\bar{A}_{\tau,\sigma}(A)=\sum_{m=0}^{r-1}\bar{A}_{\tau,\sigma}^{[m]}(A)h^{m}+\mathcal{O}(h^{r}).
\end{equation*}
Then let $\delta=\tilde{u}(\sigma)-y_{0}$ and we have
\begin{equation*}
\delta=
\tilde{u}^{[0]}(\sigma)-y_{0}+\mathcal{O}(h)=y_{0}-y_{0}+\mathcal{O}(h)=\mathcal{O}(h).
\end{equation*}
We expand  $f(\tilde{u}(\sigma))$ at $y_{0}$ and insert  the above
equalities into the first equation of the ECr method \eqref{EEPCr}.
This manipulation yields
\begin{equation}\label{comp}
\begin{aligned}
&\sum_{m=0}^{r-1}\tilde{u}^{[m]}(\tau)h^{m}=\sum_{m=0}^{r-1}
\frac{\tau^m A^my_{0}}{m!} h^{m}\\
&+\tau
h\int_{0}^{1}\sum_{k=0}^{r-1}\bar{A}_{\tau,\sigma}^{[k]}(A)h^{k}\sum_{n=0}^{r-1}\frac{1}{n!}g^{(n)}(y_{0})
 (\underbrace{\delta,\ldots,\delta}_{n-fold})d\sigma+\mathcal{O}(h^{r}).
\end{aligned}\end{equation}
 In order to show that $\tilde{u}(\tau)$ is regular, we need only to
prove that
$$\tilde{u}^{[m]}(\tau)\in
P_{m}^{d}=\underbrace{P_{m}([0,1])\times\ldots\times
P_{m}([0,1])}_{d-fold}\ \ \ \textmd{for}\  \ \ m=0,1,\ldots,r-1,$$
 where  $P_{m}([0,1])$ consists of polynomials of degrees $\leq m$ on
$[0,1]$. This can be confirmed by induction as follows.

Firstly, it is clear that $\tilde{u}^{[0]}(\tau)=y_{0}\in
P_{0}^{d}$. We assume that $\tilde{u}^{[n]}(\tau)\in P_{n}^{d}$ for
$n=0,1,\ldots,m$. Comparing the coefficients of $h^{m+1}$ on both
sides of \eqref{comp} and using \eqref{Aexplicit}  lead  to
\begin{equation*}
\begin{aligned}
&\tilde{u}^{[m+1]}(\tau)=\frac{\tau^{m+1} A^{m+1} }{(m+1)!} y_{0}+
\sum_{k+n=m}\tau\int_{0}^{1}\bar{A}_{\tau,\sigma}^{[k]}(A)h_{n}(\sigma)d\sigma\\
&=\frac{\tau^{m+1} A^{m+1} }{(m+1)!} y_{0}+
\sum_{k+n=m}\tau\int_{0}^{1}\int_{0}^{1}\Big[e^{(1-\xi)\tau h
A}P_{\xi \tau,\sigma}\Big]^{[k]} h_{n}(\sigma)d\sigma d\xi,\quad
h_{n}(\sigma)\in P_{n}^{d}.
\end{aligned}
\end{equation*}
Since $P_{\xi \tau,\sigma}$ is regular, it is easy to check that
$e^{(1-\xi)\tau h A}P_{\xi \tau,\sigma}$ is also regular. Thus,
under the condition $k+n=m$,  we have
$$\int_{0}^{1}\Big[e^{(1-\xi)\tau h
A}P_{\xi \tau,\sigma}\Big]^{[k]}
h_{n}(\sigma)d\sigma:=\check{p}^{k}_{m}(\xi \tau) \in
P^d_{m}([0,1]).$$ Then,
 the
above result can be simplified as
\begin{equation*}
\begin{aligned}
\tilde{u}^{[m+1]}(\tau) &=\frac{\tau^{m+1} A^{m+1} }{(m+1)!}
y_{0}+\sum_{k+n=m}\tau\int_{0}^{1}\check{p}^{k}_{m}(\xi \tau) d\xi \\
&=\frac{\tau^{m+1} A^{m+1} }{(m+1)!}
y_{0}+\sum_{k+n=m}\int_{0}^{\tau}\check{p}^{k}_{m}(\alpha)d\alpha\in P_{m+1}^{d}.\\
\end{aligned}
\end{equation*}
\end{proof}

  According to Lemma 3.4 presented in \cite{Li_Wu(na2016)} and
the above lemma, we get
\begin{equation}\label{error}
\mathcal{P}_{h}g(\tilde{u}(\tau))-g(\tilde{u}(\tau))=\mathcal{O}(h^{r}),
\end{equation}
which will be used in the analysis of algebraic order.
%
We are now ready to present the result about the algebraic order of
the ECr method \eqref{EEPCr}.
\begin{theo}\label{order}
About the stage order and order of the ECr method \eqref{EEPCr}, we
have
\begin{equation*} \begin{aligned}&\tilde{u}(\tau)-y(t_{0}+\tau
h)=\mathcal{O}(h^{r+1}),\ \  0<\tau<1,\\
&\tilde{u}(1)-y(t_{0}+h)=\mathcal{O}(h^{2r+1}).\end{aligned}
\end{equation*}
\end{theo}
\begin{proof} According to the previous preliminaries,
 we obtain
\begin{equation}\label{stage}
\begin{aligned}
&\tilde{u}(\tau)-y(t_{0}+\tau h)=y(t_{0}+\tau h,t_{0}+\tau h,\tilde{u}(\tau))-y(t_{0}+\tau h,t_{0},y_{0})\\
&=\int_{0}^{\tau}\frac{d}{d\alpha}y(t_{0}+\tau h,t_{0}+\alpha h,\tilde{u}(\alpha))d\alpha\\
&=\int_{0}^{\tau}(h\frac{\partial y}{\partial\tilde{t}}(t_{0}+\tau
h,t_{0}+\alpha h,\tilde{u}(\alpha))+
\frac{\partial y}{\partial\tilde{y}}(t_{0}+\tau h,t_{0}+\alpha h,\tilde{u}(\alpha))h\tilde{u}^{\prime}(\alpha))d\alpha\\
&=\int_{0}^{\tau}\Big(-h\frac{\partial
y}{\partial\tilde{y}}(t_{0}+\tau h,t_{0}+\alpha
h,\tilde{u}(\alpha))\big(A\tilde{u}(\alpha)+g(\tilde{u}(\alpha))\big)\\
&\quad+ \frac{\partial y}{\partial\tilde{y}}(t_{0}+\tau
h,t_{0}+\alpha h,\tilde{u}(\alpha))\big(hA
\tilde{u}(\alpha)+h\langle
P_{\tau,\sigma},g(\tilde{u}(\alpha))\rangle_{\alpha}\big)\Big)d\alpha\\
&=-h\int_{0}^{\tau}\Phi^{\tau}(\alpha)\big(g(\tilde{u}(\alpha))-\mathcal{P}_{h}(g\circ
\tilde{u})(\alpha)\big)d\alpha=\mathcal{O}(h^{r+1}),
\end{aligned}
\end{equation}
where $\Phi^{\tau}(\alpha)=\frac{\partial
y}{\partial\tilde{y}}(t_{0}+\tau h,t_{0}+\alpha
h,\tilde{u}(\alpha)).$ Letting $\tau=1$ in \eqref{stage} yields
\begin{equation}\label{inter}
\begin{aligned}
&\tilde{u}(1)-y(t_{0}+h)=-h\int_{0}^{1}\Phi^{1}(\alpha)\big(g(\tilde{u}(\alpha))-\mathcal{P}_{h}(g\circ \tilde{u})(\alpha)\big)d\alpha.\\
\end{aligned}
\end{equation}
 We partition the matrix-valued function $\Phi^{1}(\alpha)$  as
$\Phi^{1}(\alpha)=(\Phi_{1}^{1}(\alpha),\ldots,\Phi_{d}^{1}(\alpha))^{\intercal}$.
It follows from Lemma \ref{lemma1} that
\begin{equation}\label{pre1}
\Phi_{i}^{1}(\alpha)=\mathcal{P}_{h}\Phi_{i}^{1}(\alpha)+\mathcal{O}(h^{r}),\quad
i=1,\ldots,d.
\end{equation}
On the other hand, 
we have
\begin{equation}\label{pre2}
\begin{aligned}
&\int_{0}^{1}(\mathcal{P}_{h}\Phi_{i}^{1}(\alpha))^{\intercal}
g(\tilde{u}(\alpha)) d\alpha
=\int_{0}^{1}(\mathcal{P}_{h}\Phi_{i}^{1}(\alpha))^{\intercal}
\mathcal{P}_{h}(g\circ \tilde{u})(\alpha) d\alpha,\quad
i=1,\ldots,d.
\end{aligned}\end{equation}
 Therefore, from \eqref{inter}, \eqref{pre1} and \eqref{pre2},
it follows that
\begin{equation*}
\begin{aligned}
&\tilde{u}(1)-y(t_{0}+h)
= -h\int_{0}^{1}\left(\left(\begin{array}{c}(\mathcal{P}_{h}\Phi_{1}^{1}(\alpha))^{\intercal}\\
\vdots\\(\mathcal{P}_{h}\Phi_{d}^{1}(\alpha))^{\intercal}\end{array}\right)+\mathcal{O}(h^{r})\right)\big(g(\tilde{u}(\alpha))-\mathcal{P}_{h}(g\circ \tilde{u})(\alpha)\big)d\alpha\\
=&-h\int_{0}^{1}\left(\begin{array}{c}(\mathcal{P}_{h}\Phi_{1}^{1}(\alpha))^{\intercal}\big(g(\tilde{u}(\alpha))-\mathcal{P}_{h}(g\circ \tilde{u})(\alpha)\big)\\
\vdots\\(\mathcal{P}_{h}\Phi_{d}^{1}(\alpha))^{\intercal}\big(g(\tilde{u}(\alpha))-\mathcal{P}_{h}(g\circ
\tilde{u})(\alpha)\big)\end{array}\right)d\alpha
-h\int_{0}^{1}\mathcal{O}(h^{r})\times\mathcal{O}(h^{r})d\alpha\\
=&0+\mathcal{O}(h^{2r+1})=\mathcal{O}(h^{2r+1}).\\
\end{aligned}
\end{equation*}
\end{proof}

\section{Application in stiff gradient systems}\label{gradient systems}
When the matrix $Q$ in \eqref{IVPPP} is identity matrix, then the
system \eqref{IVPPP} is  a stiff gradient system as follows:
\begin{equation}\label{IVPPP-G}
y'=- \nabla U(y),\quad y(0)=y_{0}\in\mathbb{R}^{d},\quad t\in[0,T],
\end{equation}
where the potential $U$   has the form
\begin{equation}\label{spe-U}
  U(y)=\dfrac{1}{2}y^{\intercal}My+V(y).
\end{equation}
 Such problems arise from the
spatial discretisation of Allen--Cahn and Cahn--Hilliard PDEs (see,
e.g.  \cite{Barrett02}).   Along every exact solution, it is true
that
\begin{equation*}
 \frac{d}{dt}U(y(t))= \nabla
 U(y(t))^{\intercal}y'(t)=-y'(t)^{\intercal} y'(t)\leq0,
\end{equation*}
which  implies  that $U(y(t))$ is monotonically decreasing.

For solving this stiff gradient system, it follows from Theorem
\ref{DP} that our practical ECr method \eqref{PEEPCr} is
unconditionally energy-diminishing. For a quadratic potential (i.e.,
$V(y)=0$ in \eqref{spe-U}), the numerical solution of our method is
given by
\begin{equation*}  y_{1}=R(-hA)y_0=e^{ -h  A}y_{0}.
\end{equation*}
The importance of the damping property $ |R(\infty)| < 1$ for the
approximation properties of   Runge-Kutta methods has been studied
and  well understood in \cite{Lubich93,Lubich96} for solving
semilinear parabolic equations. The role of the condition  $
|R(\infty)| < 1$  in the approximation of stiff differential
equations has been researched in Chapter VI of
\cite{Hairer93-stiff}. It has been shown in \cite{hairer2013} that
for  each  Runge-Kutta method the energy decreases  once the
stepsize satisfies some conditions. Discrete-gradient methods, AVF
methods and AVF collocation methods derived in \cite{hairer2013} are
unconditionally energy-diminishing methods but they show no damping
for very stiff gradient systems.
  However,  it is   clear  that  our methods are unconditionally energy-diminishing methods and they have  $$ |R(\infty)| =|e^{
-\infty}|=0.$$ This  implies  that our methods are strongly damped
even for very stiff gradient systems and this is a significant
feature.

\section{Practical examples of the new methods}\label{methods}
In this  section,  we present the practical examples of the new
methods.   By the choice of $\tilde{\varphi}_{k}(\tau)=(\tau h)^k$
for $k=0,1,\ldots,r-1$ and using the Gram-Schmide process, we obtain
the standard orthonormal basis
 of $Y_{h}$ as
\begin{equation*}
\hat{p}_j(\tau)=(-1)^j\sqrt{2j+1}\sum\limits_{k=0}^ {j}{j
\choose{k}}{j+k \choose{k}}(-\tau)^k,\qquad j=0,1,\ldots,r-1,\qquad
\tau\in[0,1],
\end{equation*}
which are the shifted Legendre polynomials   on $[0,1]$. Therefore,
$P_{\tau,\sigma}$ can be determined by its limit as follows $
P_{\tau,\sigma}= \lim_{h\to 0}P_{\tau,\sigma}
=\sum_{i=0}^{r-1}\hat{p}_{i}(\tau)\hat{p}_{i}(\sigma), $
\subsection{An example of ECr methods} For the ECr method \eqref{EEPCr}, we need to calculate
$\bar{A}_{\tau,\sigma}(A)$ appearing in the methods. It follows form
\eqref{Aexplicit} that
\begin{equation}\label{barA}
\begin{aligned}
&\bar{A}_{\tau,\sigma}(A)=\int_{0}^1 e^{(1-\xi)\tau h A}P_{\xi
\tau,\sigma}d\xi=\sum_{i=0}^{r-1}
 \int_{0}^1 e^{(1-\xi)\tau h A}\hat{p}_{i}(\xi \tau)d\xi
 \hat{p}_{i}(\sigma)\\
 =&\sum_{i=0}^{r-1}
 \int_{0}^1 e^{(1-\xi)\tau h A}(-1)^i\sqrt{2i+1}\sum\limits_{k=0}^ {i}{i \choose{k}}{i+k
\choose{k}}(-\xi \tau)^k d\xi \hat{p}_{i}(\sigma)\\
=&\sum_{i=0}^{r-1}\sqrt{2i+1}\sum\limits_{k=0}^
{i}(-1)^{i+k}\frac{(i+k)!}{k!(i-k)!}\bar{\varphi}_{k+1}(\tau hA)
\hat{p}_{i}(\sigma).\end{aligned}
\end{equation}
Here  the $\bar{\varphi}$-functions (see, e.g.
\cite{Hochbruck1998,Hochbruck2005,Hochbruck2010,Hochbruck2009}) are
defined by:
\begin{equation*}
\bar{\varphi}_0(z)=e^{z},\ \ \bar{\varphi}_k(z)=\int_{0}^1
e^{(1-\sigma)z}\frac{\sigma^{k-1}}{(k-1)!}d\sigma, \ \ k=1,2,\ldots.
\label{phi}%
\end{equation*}
It is noted that  a number of approaches have been developed which
work with the application of the $\varphi$-functions on a vector
(see \cite{Higham2011,Hochbruck1997,Hochbruck2010}, for example).
\subsection{An example of TCr methods} For the TCr
method \eqref{EEPCr of erkn} of solving $q^{\prime\prime}(t)+\Omega
q(t)=-\nabla U(q(t))$, we need to compute
$\mathcal{A}_{\tau,\sigma}$ and $\mathcal{B}_{1,\sigma}$.  It
follows from \eqref{AB of erkn} that
\begin{equation*}
\begin{aligned} &\mathcal{A}_{\tau,\sigma}(K)= \sum_{j=0}^{r-1}
 \int_{0}^1(1-\xi) \phi_{1} \big((1-\xi)^2K\big)\hat{p}_{j}(\xi
 \tau)d\xi \hat{p}_{j}(\sigma)\\
 =&\sum_{j=0}^{r-1}\sqrt{2j+1}\sum\limits_{l=0}^{\infty}(-1)^j\sum\limits_{k=0}^
{j}{j \choose{k}}{j+k \choose{k}}\int_{0}^1(-\xi )^k(1-\xi )^{2l+1}d\xi \dfrac{(-1)^{l}K^{l}}{(2l+1)!}\tau^k\hat{p}_{j}(\sigma)\\
=&\sum_{j=0}^{r-1}\sqrt{2j+1}\sum\limits_{l=0}^{\infty}\sum\limits_{k=0}^
{j}(-1)^{j+k}{j \choose{k}}{j+k \choose{k}}\dfrac{k!(2l+1)!}{(2l+k+2)!}\dfrac{(-1)^{l}K^{l}}{(2l+1)!}\tau^k\hat{p}_{j}(\sigma)\\
=&\sum_{j=0}^{r-1}\sqrt{2j+1}\hat{p}_{j}(\sigma)\sum\limits_{l=0}^{\infty}\sum\limits_{k=0}^
{j} \dfrac{(-1)^{j+k+l}(j+k)!}{k!(j-k)!(2l+k+2)!}\tau^k K^{l}.
 \end{aligned}
\end{equation*}
Recall that the generalized hypergeometric function ${}_mF_{n}$  is
defined by \begin{equation}
\begin{aligned} {}_mF_{n}\left[\begin{matrix}\alpha_1,\alpha_2,\ldots,\alpha_m;\\ \beta_1,\beta_2,\ldots,\beta_n;\end{matrix}
x\right]=\sum\limits_{l=0}^{\infty}\dfrac{\prod_{i=1}^{m}(\alpha_i)_l}{\prod_{i=1}^{n}(\beta_i)_l}\dfrac{x^l}{l!},
\end{aligned}
\label{hypergeometric}%
\end{equation}
where $\alpha_i$ and $\beta_i$ are arbitrary complex numbers, except
that $\beta_i$  can be neither zero nor  a negative
 integer, and  $(z)_l$ is the   Pochhammer symbol which is defined as
$$(z)_0=1,\ \ (z)_l=z(z+1)\cdots(z+l-1),\ \ \ l\in \mathbb{N}.$$
Then, $\mathcal{A}_{\tau,\sigma}$ can be expressed  by
\begin{equation}\label{A re of erkn}
\begin{aligned} &\mathcal{A}_{\tau,\sigma}(K)
=\sum_{j=0}^{r-1}\sqrt{2j+1}\hat{p}_{j}(\sigma)\sum\limits_{l=0}^{\infty}\frac{(-1)^{j+l}}{(2l+2)!}{}_2F_{1}\left[\begin{matrix}-j,j+1;\\
2l+3;\end{matrix} \tau\right]K^{l}.\\
 \end{aligned}
\end{equation}

Likewise, we can obtain
 \begin{equation}\label{B}%
\begin{aligned} &\mathcal{B}_{1,\sigma}(K)=
\sum_{j=0}^{r-1}\sqrt{2j+1}\hat{p}_{j}(\sigma)S_{j}(K),
 \end{aligned}
\end{equation}
where $S_{j}(K)$ are given by
\begin{equation}
\begin{aligned}
S_{2j}(K)
=&(-1)^j \frac{(2j)!}{(4j+1)!}K^j{}_0F_{1}\left[\begin{matrix}-;\\
 \frac{1}{2};\end{matrix}  -\frac{K}{16}\right]{}_0F_{1}\left[\begin{matrix}-;\\
2j+\frac{3}{2};\end{matrix}  -\frac{K}{16}\right],\\
S_{2j+1}(K)
=&(-1)^{j} \frac{(2j+2)!}{(4j+4)!}K^{j+1} {}_0F_{1}\left[\begin{matrix}-;\\
 \frac{3}{2};\end{matrix}  -\frac{K}{16}\right]{}_0F_{1}\left[\begin{matrix}-;\\
2j+\frac{5}{2};\end{matrix}  -\frac{K}{16}\right],\quad j=0,1,\ldots.\\
\end{aligned}
\label{Computate S}%
\end{equation}

%

\subsection{An example of RKNCr methods} By letting
$K=0$ in the above analysis, we obtain an example of RKNCr methods
for solving the general second-order ODEs \eqref{ge2oprob} as
\begin{equation*}
\left\{\begin{aligned} q_{d_i}=& q_0 +d_i h p_0 - d_i ^2 h^2
\int_{0}^1 \mathcal{\bar{A}}_{d_i,\sigma}
\nabla U\Big(\sum_{m=1}^{r}q_{d_m}l_{m}(\sigma)\Big)d\sigma, \ \   i=1,\ldots,r,\\
q_{1}=& q_0 +  h p_0 -  h^2 \int_{0}^1 \mathcal{\bar{A}}_{1,\sigma}
\nabla U\Big(\sum_{m=1}^{r}q_{d_m}l_{m}(\sigma)\Big)d\sigma, \\
p_{1}= &  p_0-h  \int_{0}^1
 \mathcal{\bar{B}}_{1,\sigma}   \nabla U\Big(\sum_{m=1}^{r}q_{d_m}l_{m}(\sigma)\Big)d\sigma,\\
\end{aligned}\right.
\end{equation*}
where $\bar{\mathcal{A}}_{\tau,\sigma} = \sum_{i=0}^{r-1}
 \int_{0}^1(1-\xi) \hat{p}_{i}(\xi
 \tau)d\xi \hat{p}_{i}(\sigma)\ \textmd{and}\ \bar{\mathcal{B}}_{1,\sigma}= \sum_{i=0}^{r-1}
 \int_{0}^1  \hat{p}_{i}(\xi  )d\xi
\hat{p}_{i}(\sigma). $

\begin{rem}
It is noted that one can make  different choices of  $Y_h$ and $X_h$
and the whole analysis presented in this paper still holds.
Different choices will produce different practical methods, and in
this paper, we do not go further on this point for brevity.
\end{rem}

\section{Numerical experiments}\label{Numerical experiments}
Applying the $r$-point Gauss--Legendre's quadrature  to the integral
of
 \eqref{EEPCr} yields
\begin{equation}\label{PEEPCr}
\left\{\begin{aligned} &y_{c_{i}}=e^{c_{i} h
A}y_{0}+c_{i}h\sum_{j=1}^{r}b_j
  \bar{A}_{c_{i},c_j}(A)
g(y_{c_{j}}),\quad i=1,\ldots,r,\\
&y_{1}=e^{  h A}y_{0}+ h \sum_{j=1}^{r}b_j
  \bar{A}_{1,c_j}(A)
g(y_{c_{j}}),
\end{aligned}\right.
\end{equation}
where $c_{j}$ and $ b_{j}$ with $j=1,\ldots,r$ are the   nodes and
weights of the quadrature, respectively. In order to  show the
efficiency and robustness of our methods, we take $r=2$  and denote
the corresponding method by EC2P.    Then we choose the same $Y_{h}$
and $X_{h}$  for the functionally fitted energy-preserving method
developed in \cite{Li_Wu(na2016)}, and by this choice, the method
becomes the $2r$th order RKEPC method given in \cite{Hairer2010}.
For this method, we choose $r=2$ and approximate the integral  by
the Lobatto quadrature of order eight, which is precisely the
``extended Labatto IIIA method of order four" in
\cite{Iavernaro2009}. We denote the method as RKEPC2. Another
integrator we select for comparisons is the explicit three-stage
exponential integrator of order four derived in \cite{Hochbruck2009}
which is denoted by EEI3s4.  It is noted  that the first two methods
are   implicit and we
  set $10^{-16}$ as the error tolerance
 and $5$ as the maximum number of each fixed-point iteration. It is also remarked that in this paper, we only
 demonstrate the efficiency of ECr methods  when
 applied to  first-order
 systems,  for brevity.  Numerical comparisons   of TCr and RKNCr methods for solving  second-order
 highly oscillatory systems will be presented elsewhere.

\vskip2mm\noindent\textbf{Problem 1.} Consider the Duffing equation
defined by
\begin{equation*}
\begin{aligned}& \left(
                   \begin{array}{c}
                     q \\
                      p \\
                   \end{array}
                 \right)
'= \left(
    \begin{array}{cc}
      0& 1\\
       -\omega^{2}-k^{2} &0 \\
    \end{array}
  \right)\left(
                   \begin{array}{c}
                     q \\
                      p \\
                   \end{array}
                 \right)+
\left(
                                                                           \begin{array}{c}
                                                                           0\\
2k^{2}q^{3}
                                                                           \end{array}
                                                                         \right),\
                                                                         \ \left(
                                                                            \begin{array}{c}
                                                                              q(0) \\
                                                                              p(0) \\
                                                                            \end{array}
                                                                          \right)=\left(
                                                                                    \begin{array}{c}
                                                                                      0 \\
                                                                                      \omega \\
                                                                                    \end{array}
                                                                                  \right).
\end{aligned}\end{equation*}
 It is a Hamiltonian system with  the Hamiltonian:
\begin{equation*}
H(p,q)=\frac{1}{2}p^{2}+\frac{1}{2}(\omega^{2}+k^{2})q^{2}-\frac{k^{2}}{2}q^{4}.
\end{equation*}
The exact solution of this system is $q(t)=sn(\omega t;k/\omega)$
with the  Jacobi elliptic function  $sn$. Choose $k=0.07,
\omega=5,10,20$  and   solve the problem in the interval $[0, 1000]$
with different stepsizes
 $h= 0.1/2^{i}$ for $i=0,\ldots3.$  The   global errors are
presented in Figure \ref{p1-1}. Then, we integrate this problem with
 the stepsize $h=1/100$  in the interval $[0, 10000].$ See Figure
\ref{p1-2}  for the energy conservation for different methods.

 \begin{figure}[ptb]
\centering
 \includegraphics[width=4cm,height=6cm]{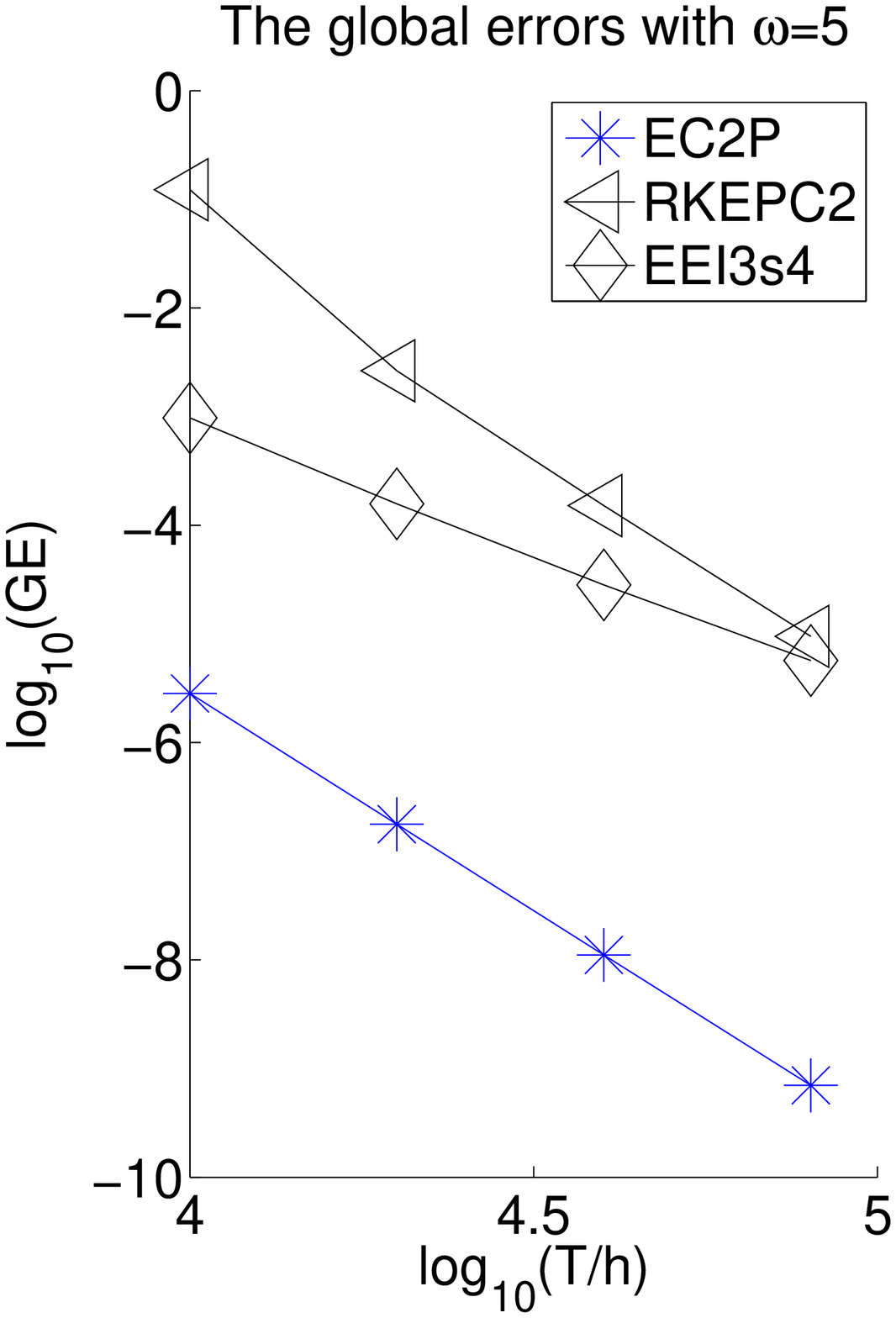}
 \includegraphics[width=4cm,height=6cm]{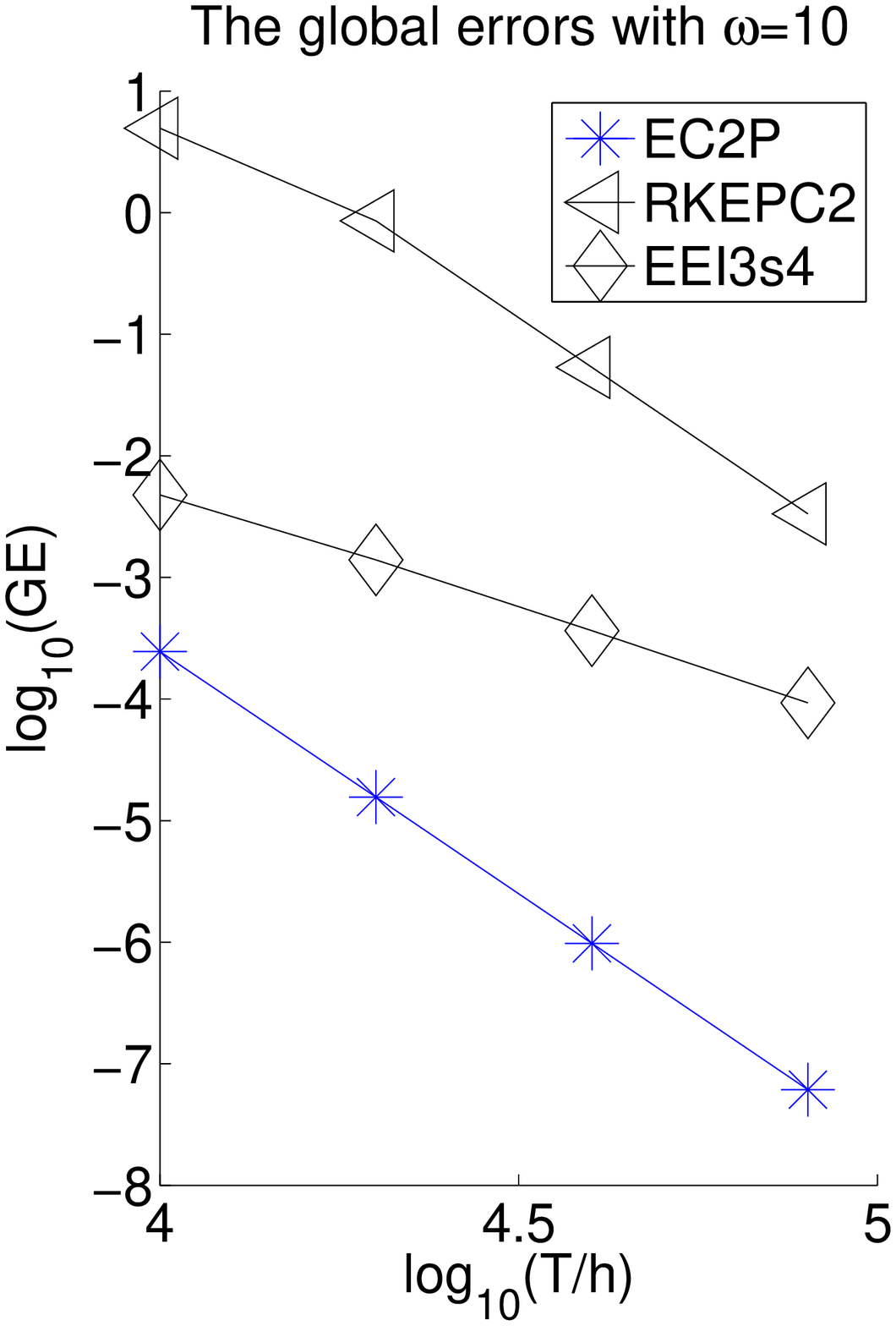}
 \includegraphics[width=4cm,height=6cm]{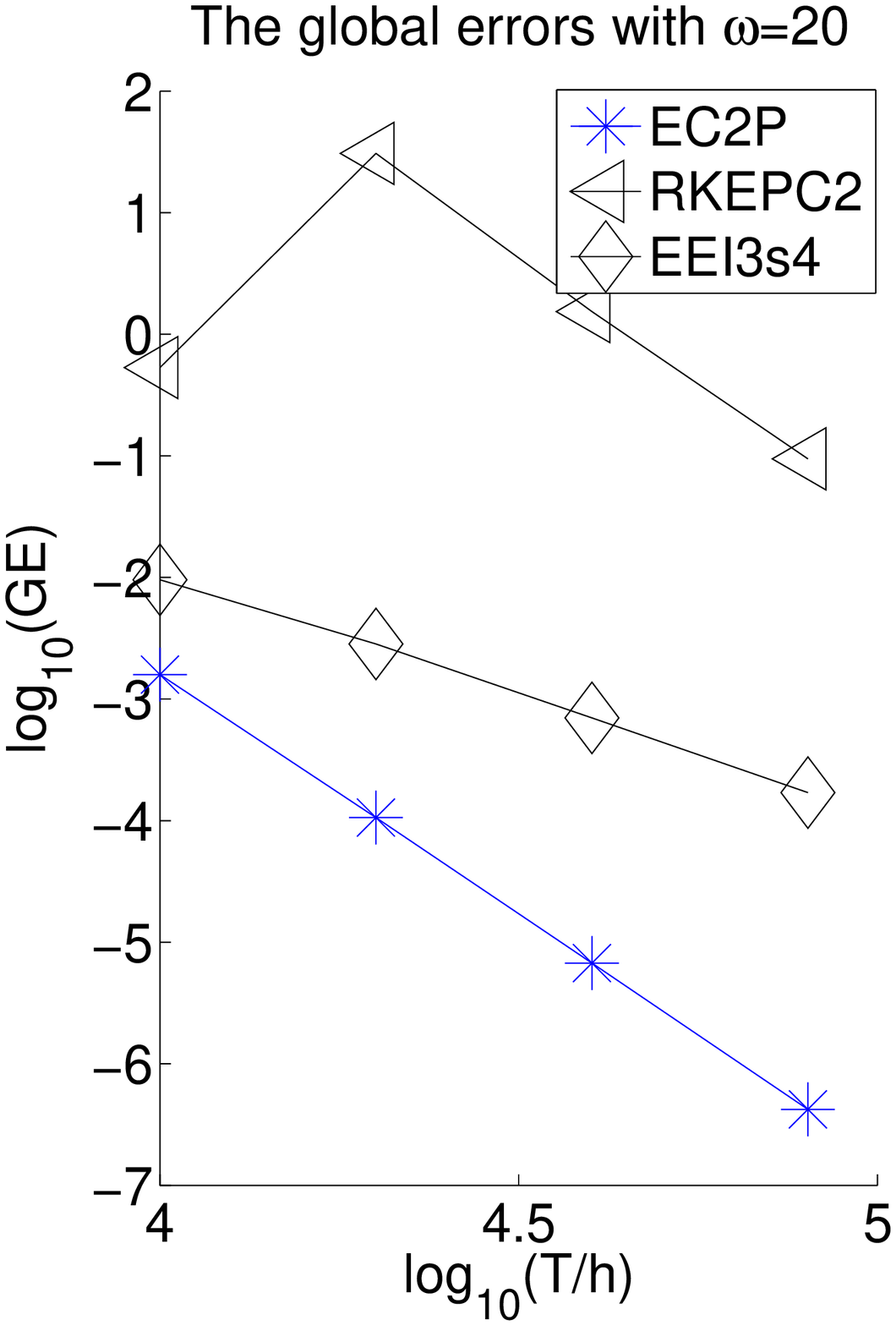}
\caption{The logarithm of the  global error against the logarithm of
$T/h$.} \label{p1-1}
\end{figure}

\begin{figure}[ptb]
\centering
\includegraphics[width=4cm,height=6cm]{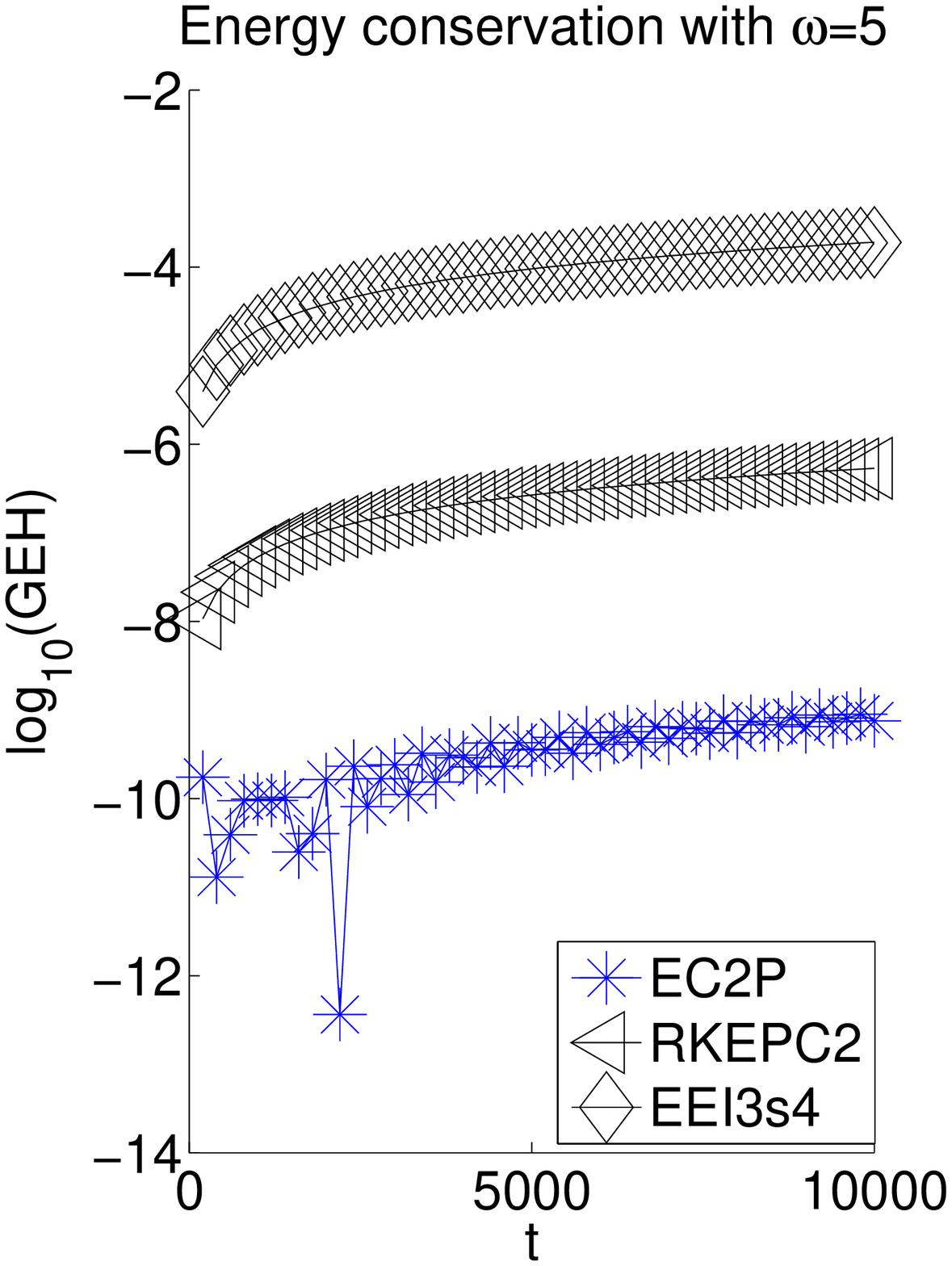}
\includegraphics[width=4cm,height=6cm]{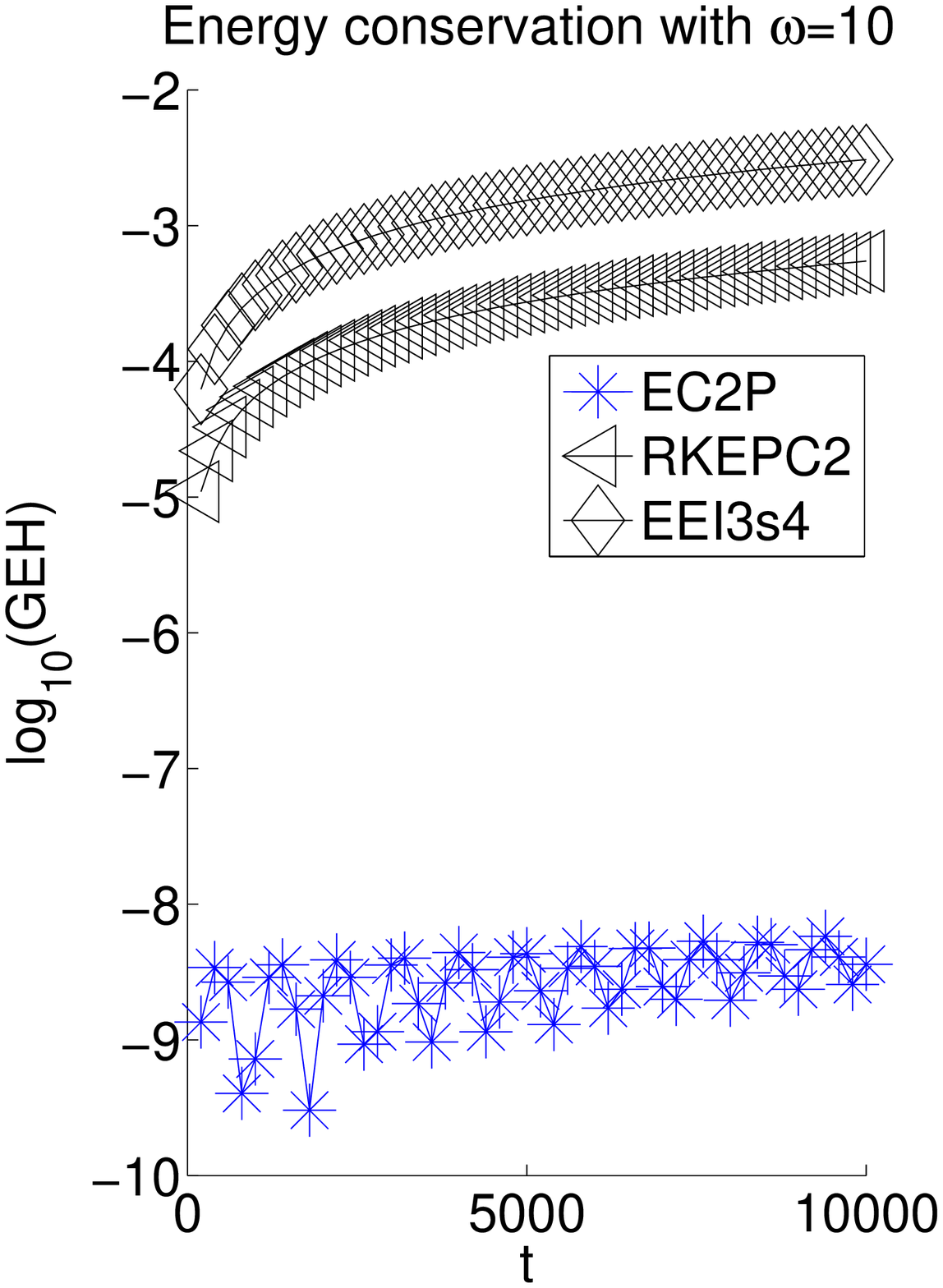}
\includegraphics[width=4cm,height=6cm]{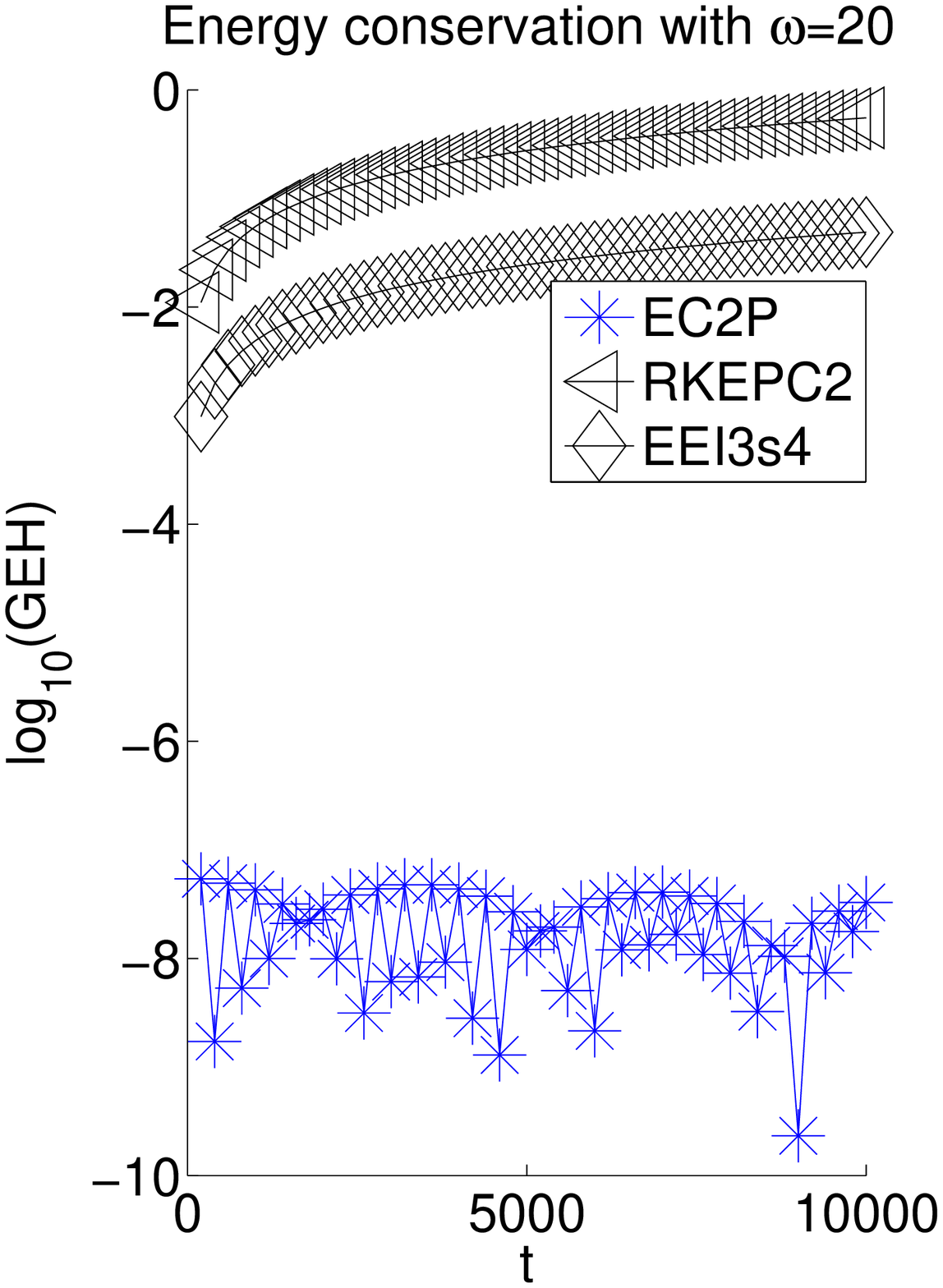}
\caption{The logarithm of the  error of Hamiltonian against $t$.}
\label{p1-2}
\end{figure}

\vskip2mm\noindent\textbf{Problem 2.}   Consider the following
averaged system in wind-induced
 oscillation (see
 \cite{Mclachlan-98})
\begin{equation*}
\begin{aligned}& \left(
                   \begin{array}{c}
                     x_1 \\
                      x_2 \\
                   \end{array}
                 \right)
'= \left(
    \begin{array}{cc}
      -\zeta& -\lambda\\
       \lambda & -\zeta \\
    \end{array}
  \right)\left(
                   \begin{array}{c}
                     x_1 \\
                      x_2 \\
                   \end{array}
                 \right)+
\left(
                                                                           \begin{array}{c}
                                                                           x_1x_2\\
\frac{1}{2}(x_1^2-x_2^2)
                                                                           \end{array}
                                                                         \right),
\end{aligned}\end{equation*}
 where $\zeta \geq
0$ is a damping factor and $\lambda$ is a detuning parameter. By
setting $$\zeta = r\cos(\theta),\qquad  \lambda =
r\sin(\theta),\qquad r \geq 0,\qquad 0 \leq\theta \leq \pi /2,$$
this system can be transformed into the scheme \eqref{IVPPP} with
\begin{equation*}
\begin{aligned}
&Q=\left(
    \begin{array}{cc}
      -\cos(\theta) & -\sin(\theta) \\
      \sin(\theta) & -\cos(\theta) \\
    \end{array}
  \right),\ \ M=\left(
                  \begin{array}{cc}
                    r & 0 \\
                    0 & r \\
                  \end{array}
                \right),\\
&V=-\frac{1}{2}\sin(\theta)\big(x_1x_2^2-\frac{1}{3}x_1^3\big)+\frac{1}{2}\cos(\theta)\big(-x_1^2x_2+\frac{1}{3}x_2^3\big).
\end{aligned}\end{equation*}
Its  first integral (conservative case, when $\theta =\pi/2$) or
Lyapunov function (dissipative case, when $\theta <\pi/2$)   is
$$H=\frac{1}{2}r(x_1^2+x_2^2)-\frac{1}{2}\sin(\theta)\big(x_1x_2^2-\frac{1}{3}x_1^3\big)+\frac{1}{2}\cos(\theta)\big(-x_1^2x_2+\frac{1}{3}x_2^3\big).$$
 The  initial values are
given by $x_1(0)=0,\ x_2(0)=1.$ Firstly we consider the conservative
case and  choose $\theta=\pi/2,\ r=20.$ The problem is integrated in
$[0,1000]$ with the stepsize $h=0.1/2^i$  for  $ i=1,\ldots,4$ and
the global errors are given  in Figure \ref{p2-1} (a). Then we solve
this system  with the stepsize $h=1/200$ in the interval $[0,
10000]$ and Figure \ref{p2-1} (b) shows the results of the energy
preservation. Secondly we  choose $\theta=\pi/2-10^{-4}$  and this
gives a dissipative system. The system is solved in $[0,1000]$ with
$h=0.1/2^i$  for  $ i=1,\ldots,4$ and the errors are presented in
Figure \ref{p2-2} (a). See Figure \ref{p2-2} (b) for the results of
Lyapunov function with $h=1/20$. Here we consider the results given
by EC2P with a  smaller stepsize $h=1/1000$ as the `exact' values of
Lyapunov function.

 \begin{figure}[ptb]
\centering
\includegraphics[width=4cm,height=6cm]{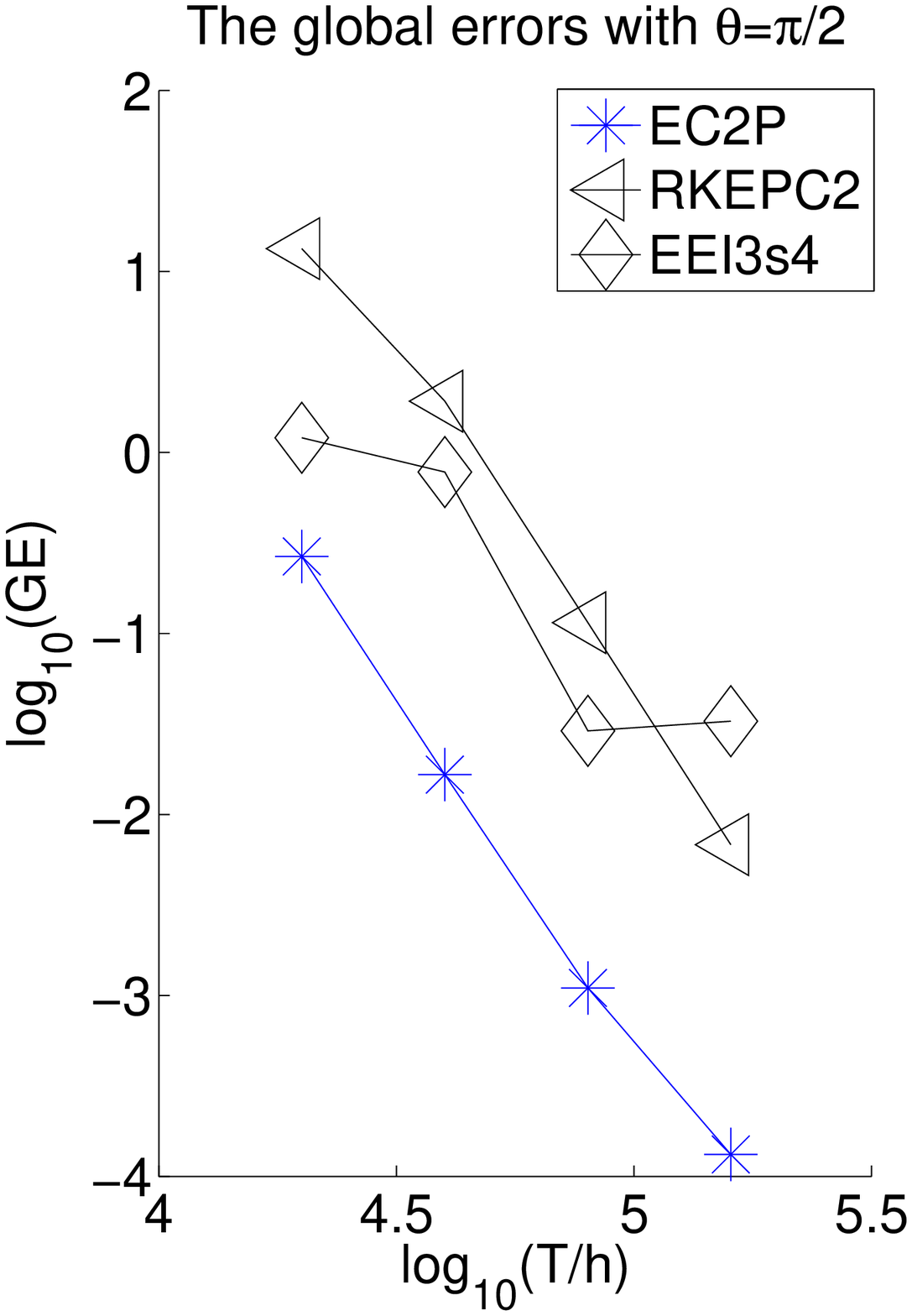}
\includegraphics[width=4cm,height=6cm]{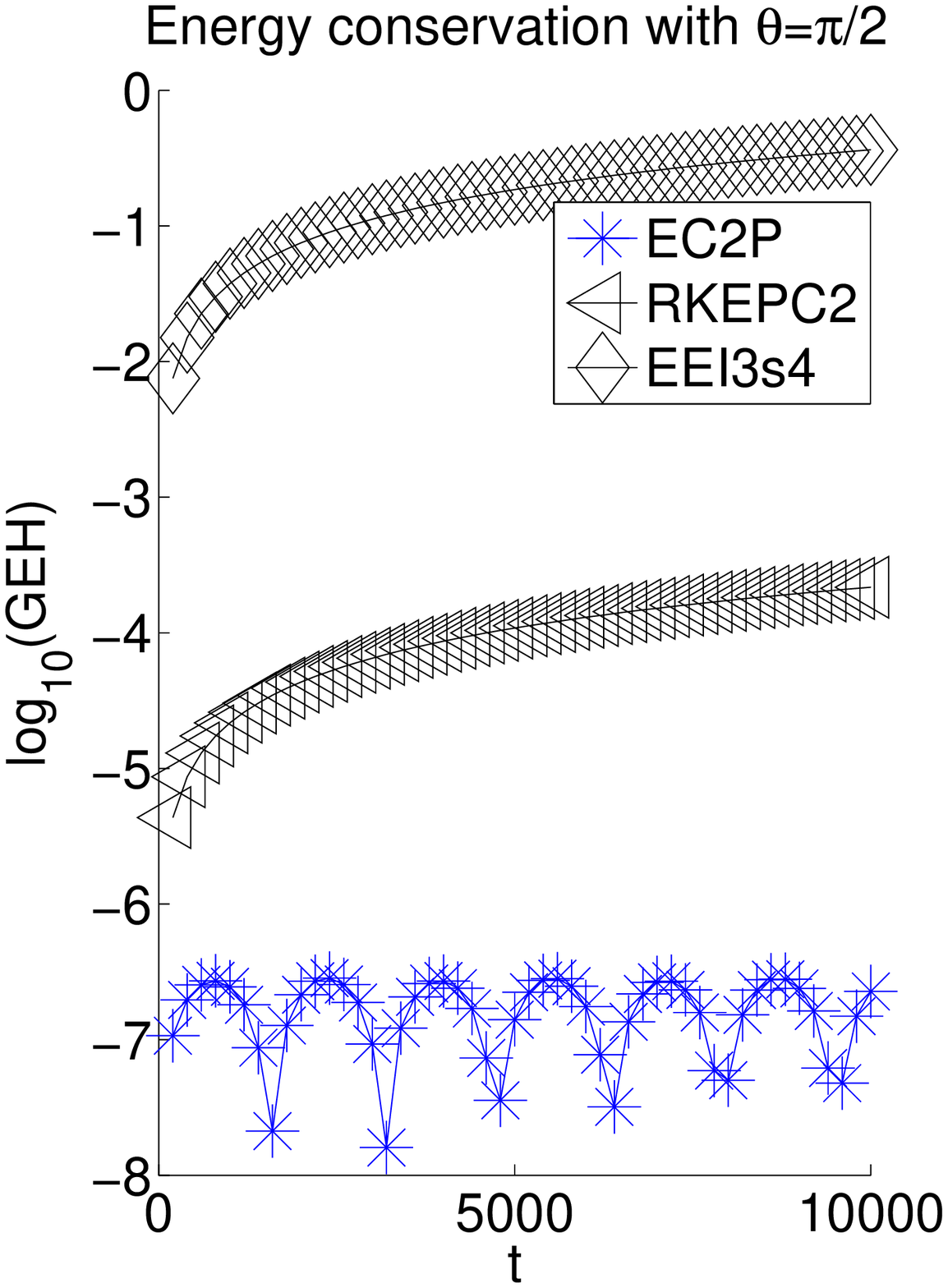}
\caption{(a) The logarithm of the  global error against the
logarithm of $T/h$. (b) The logarithm of the  error of Hamiltonian
against  $t$.} \label{p2-1}
\end{figure}

\begin{figure}[ptb]
\centering
\includegraphics[width=4cm,height=6cm]{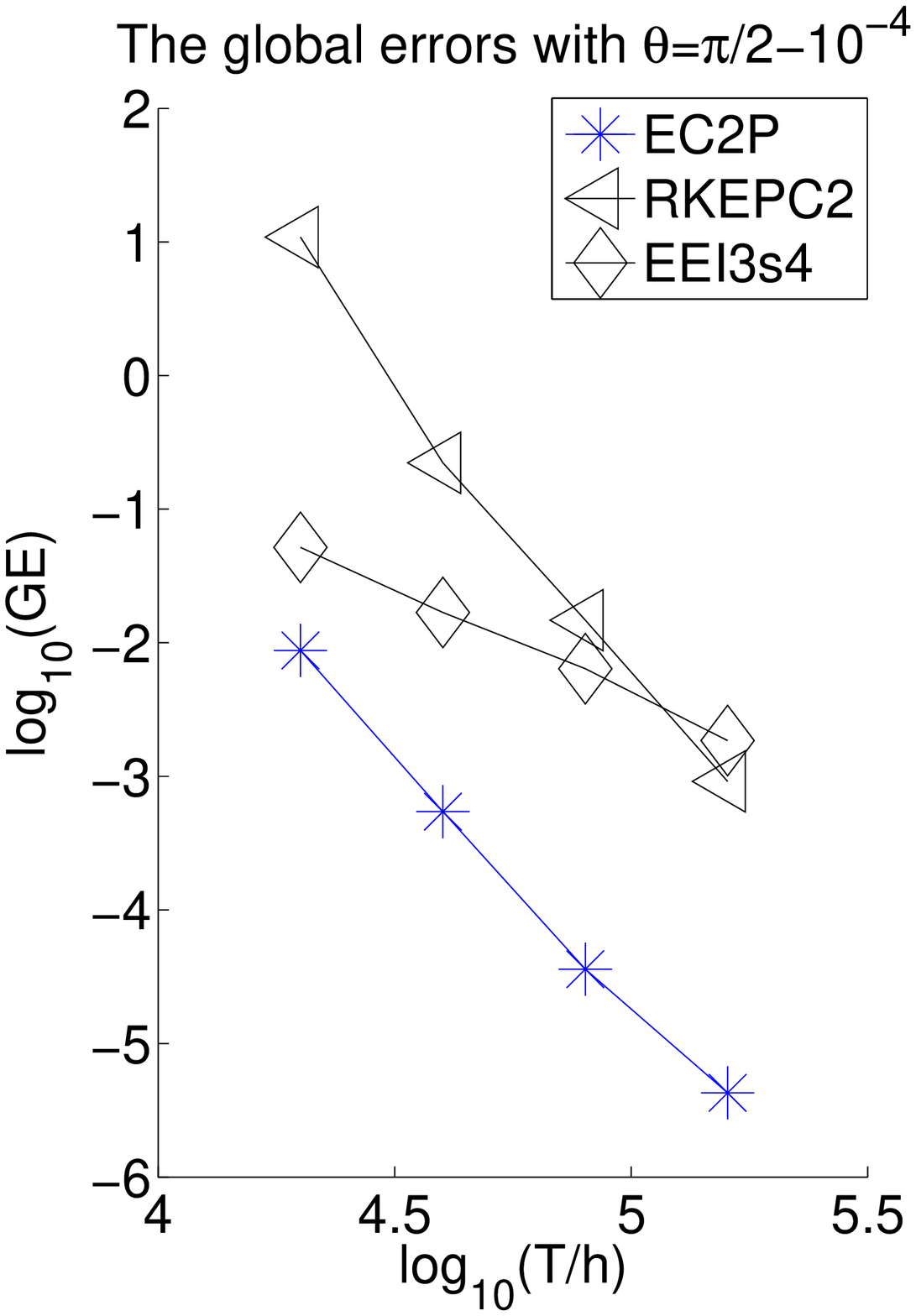}
\includegraphics[width=4cm,height=6cm]{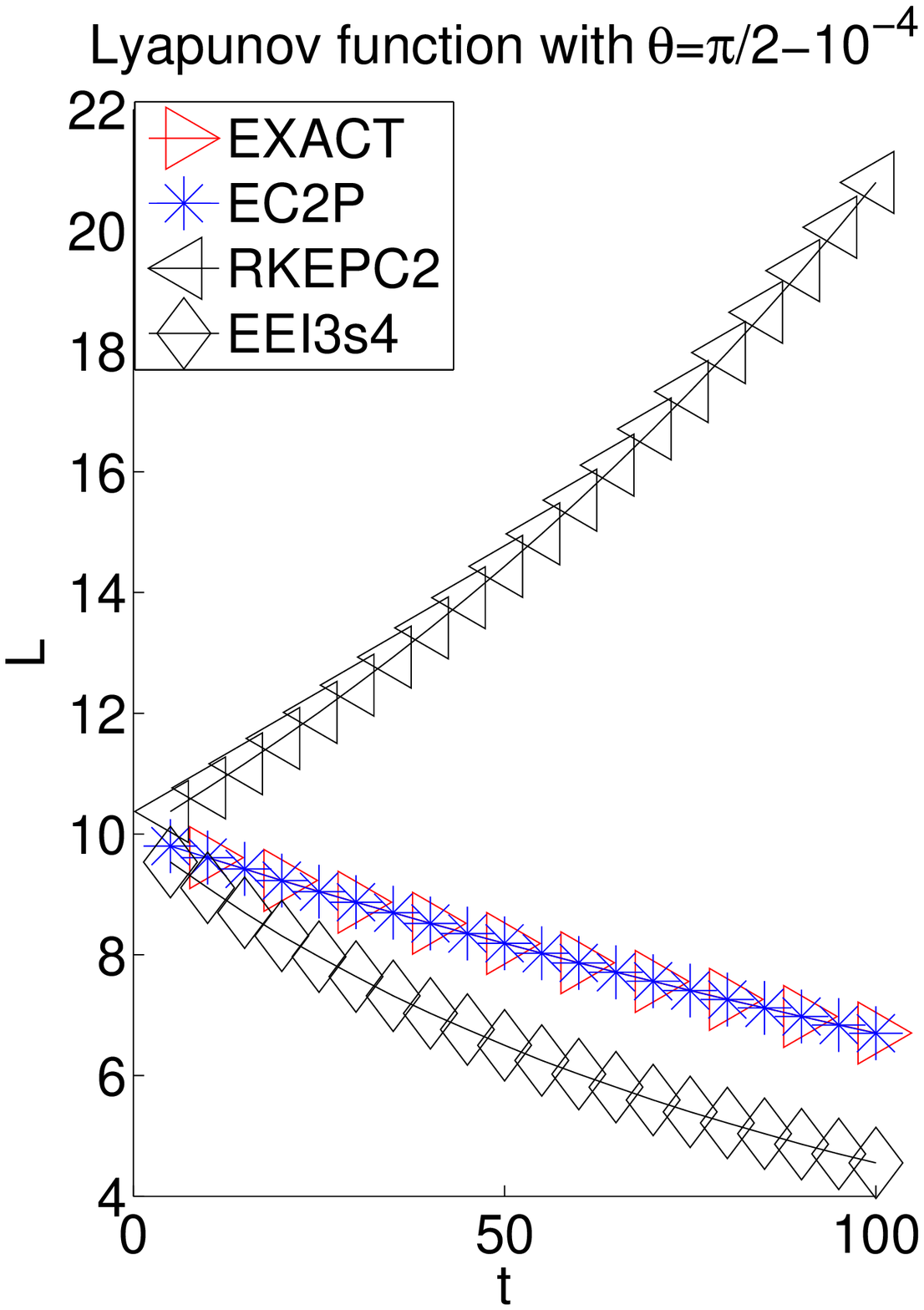}
\caption{(a) The logarithm of the  global error against the
logarithm of $T/h$. (b) The results of the Lyapunov function against
$t$.} \label{p2-2}
\end{figure}

\vskip2mm\noindent\textbf{Problem 3.} Consider  the nonlinear
Schr\"{o}dinger equation (see  \cite{Chen2001})
\begin{equation*}\begin{aligned}
&i\psi_{t}+\psi_{xx}+2|\psi|^{2}\psi=0,\quad
 \psi(x,0)= 0.5 + 0.025 \cos(\mu x)\\
\end{aligned}
\end{equation*}
with the periodic boundary condition $\psi(0,t)=\psi(L,t).$
Following \cite{Chen2001},  we choose $L =4\sqrt{2}\pi$ and $\mu =
2\pi/L.$ The initial condition chosen here is in the vicinity of the
homoclinic orbit. Using $\psi = p + \textmd{i}q,$ this equation  can
be rewritten as a pair of real-valued equations
\begin{equation*}\label{semi}
\begin{aligned}
&p_t +q_{xx} + 2(p^2  + q^2)q = 0,\\
&q_t -p_{xx} -2(p^2  + q^2)p = 0.\\
\end{aligned}
\end{equation*}
Discretising the spatial derivative $\partial_{xx}$ by the
pseudospectral method given in  \cite{Chen2001},  this problem is
converted into the following system:
\begin{equation}\label{semi sd}
\left(
  \begin{array}{c}
    \textbf{p} \\
    \textbf{q} \\
  \end{array}
\right)'=
 \begin{aligned}
 \left(
   \begin{array}{cc}
     0 & -D_2 \\
     D_2 & 0 \\
   \end{array}
 \right)\left(
  \begin{array}{c}
    \textbf{p} \\
    \textbf{q} \\
  \end{array}
\right)+\left(
          \begin{array}{c}
            -2(\textbf{p}^{2}+ \textbf{q}^{2})\cdot  \textbf{q} \\
            2( \textbf{p}^{2}+ \textbf{q}^{2})\cdot \textbf{p} \\
          \end{array}
        \right)
\end{aligned}
\end{equation}
where $\textbf{p}=(p_0,p_1,\ldots,p_{N-1})^{\intercal},\
\textbf{q}=(q_0,q_1,\ldots,q_{N-1})^{\intercal}$ and
$D_2=(D_2)_{0\leq j,k\leq N-1}$ is the pseudospectral differential
matrix defined by:
\begin{equation*}
(D_2)_{jk}=\left\{\begin{aligned}
&\frac{1}{2}\mu^2(-1)^{j+k+1}\frac{1}{\sin^2(\mu(x_j-x_k)/2)},\quad j\neq k,\\
&-\mu^2\frac{2(N/2)^2+1}{6},\quad\quad\quad\quad\quad\quad\ \ \ \ \   j=k,\\
\end{aligned}\right.
\end{equation*}
with $x_j=j\frac{L}{N}$ for  $j=0,1,\ldots,N-1.$ The Hamiltonian  of
\eqref{semi sd} is
\begin{equation*}
H(\textbf{p},\textbf{q})=\frac{1}{2}\textbf{p}^{\intercal}D_{2}\textbf{p}+\frac{1}{2}\textbf{q}^{\intercal}D_{2}\textbf{q}
+ \frac{1}{2}\sum_{i=0}^{N-1}(p_{i}^{2}+q_{i}^{2})^{2}.
\end{equation*}
 We choose $N=128$ and
first  solve the problem in the interval $[0, 10]$ with
 $h= 0.1/2^{i}$ for  $i=3,\ldots,6.$  See Figure \ref{p3} (a) for the global errors. Then,  this problem is   integrated
 with
$h=1/200$ in  $[0, 1000]$  and the  energy conservation is presented
in Figure \ref{p3} (b).

 \begin{figure}[ptb]
\centering
\includegraphics[width=4cm,height=6cm]{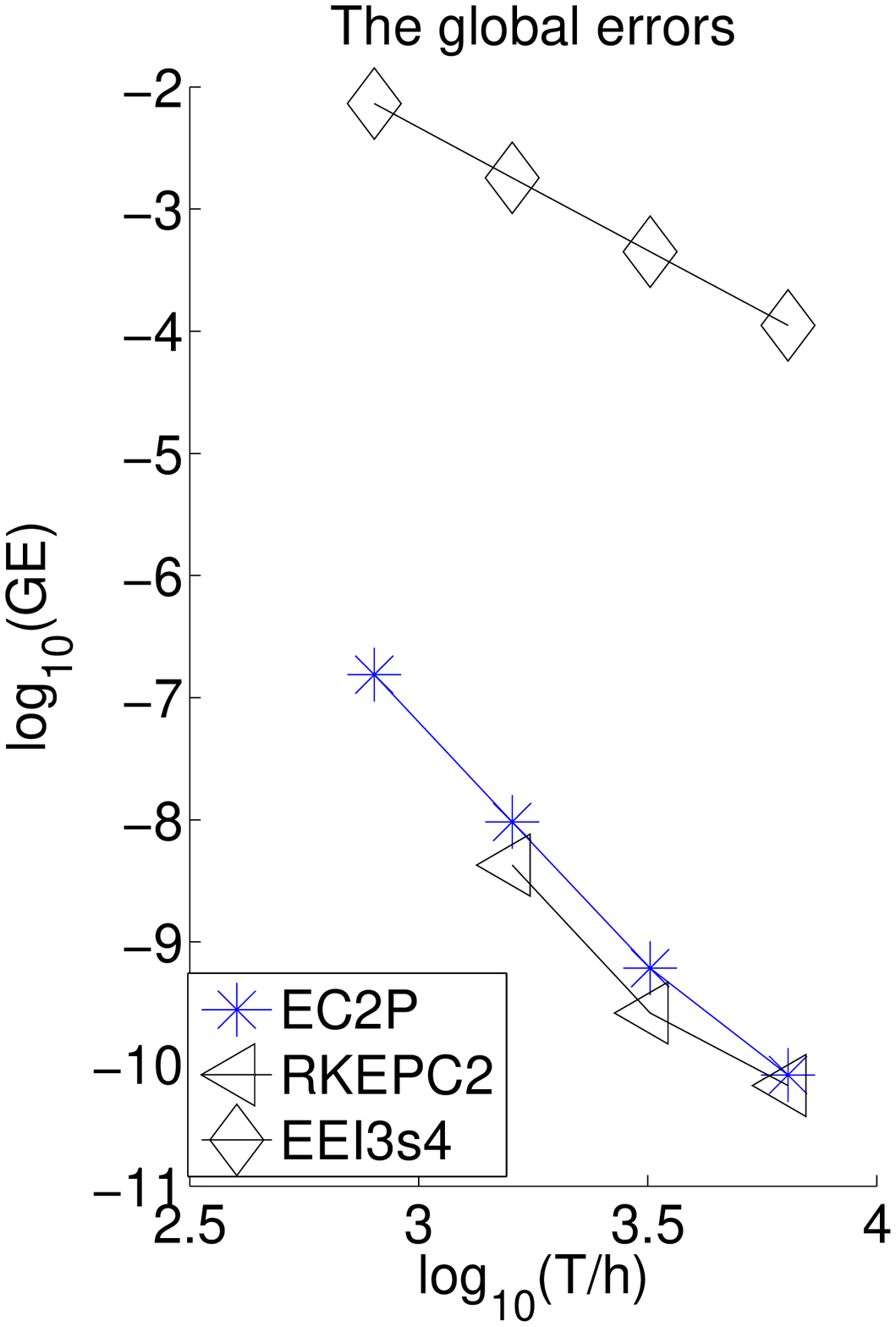}
\includegraphics[width=4cm,height=6cm]{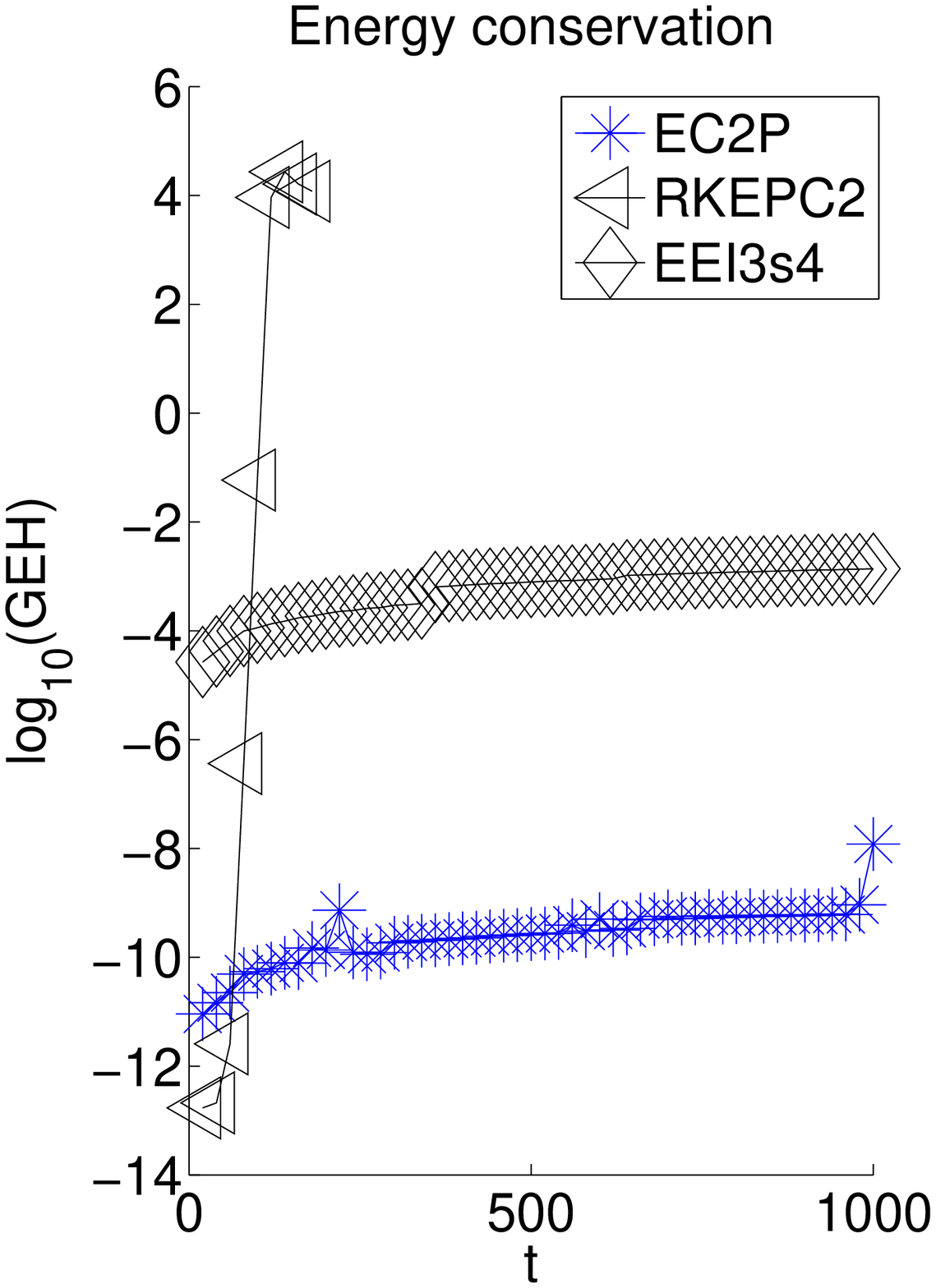}
\caption{(a) The logarithm of the  global error against the
logarithm of $T/h$. (b) The logarithm of the  error of Hamiltonian
against  $t$.} \label{p3}
\end{figure}

 It can be concluded from these numerical experiments that our
EC2P method definitely shows  higher accuracy,  more excellent
invariant-preserving property, and  prominent long-term behavior in
the numerical simulations, than the other effective methods in the
literature.

\section{Concluding remarks and discussions} \label{sec:conclusions}
 Exponential integrators  have constituted an
important class of methods for the numerical simulation of
first-order ODEs, including the semi-discrete nonlinear
Schr\"{o}dinger equation etc. Finite element methods for ODEs can
date back to early 1960s  and they have been investigated by many
researchers. In this paper, combining the ideas of these two
effective methods, we derived and analysed a new kind of exponential
collocation methods for the conservative or dissipative system
\eqref{IVPPP}. We have also rigorously analysed
 the properties including existence and uniqueness, and algebraic order. It has been proved that our novel methods can be
arbitrary-order accuracy as well as exactly or nearly preserving
first integrals or Lyapunov functions. The application of our
methods in stiff gradient systems was discussed.
 The efficiency and superiority of the new methods were numerically demonstrated by
performing some   experiments. By the analysis    of this paper,
arbitrary-order energy-preserving methods   were presented for
second-order highly oscillatory/general systems.

Last but not least, it is  noted that there are still  some issues
of the methods which can  be further considered.

\begin{itemize}\itemsep=-0.2mm

\item  The error bounds and
convergence properties of the methods   will be discussed in another
work.


\item Another issue for   exploration is the application  of our   methodology in   PDEs
such as  nonlinear Sch\"{o}rdinger equations and wave equations. We
have
  derived exponential  integrators to
preserve the continuous energy of Sch\"{o}rdinger equations (see
\cite{wang2018 S}).

\item   The application  of our   methodology in   other ODEs
such as  general gradient systems and Poisson systems will also be
considered.

\end{itemize}

\section*{Acknowledgement}

The authors are grateful to Professor Christian Lubich for his
careful reading of the manuscript and for his helpful comments. It
is also his idea that motivates Section \ref{gradient systems} of
this manuscript.

\end{document}